\documentclass[final,leqno]{siamltex704}
\usepackage{hyperref}
\usepackage{amsmath}
\usepackage{graphicx}
\usepackage{array}
\usepackage{lineno}
\usepackage[notcite,notref]{showkeys}
\usepackage{tikz}
\usepackage{amsfonts,amssymb}
\usepackage{dsfont}
\usepackage{pifont}
\def\E{{\mathcal E}}

\newtheorem{remark}{Remark}[section]

\def\O{\Omega}
\def\pa{\partial}
\def\3bar{{|\hspace{-.02in}|\hspace{-.02in}|}}
\newtheorem{gWG}{Generalized weak Galerkin scheme}

\begin{document}

\setlength{\parindent}{0.25in} \setlength{\parskip}{0.08in}

\title{ 
GENERALIZED WEAK GALERKIN FINITE ELEMENT METHODS
FOR SECOND ORDER ELLIPTIC PROBLEMS}

\author{
Dan Li\thanks{ {Jiangsu Key Laboratory for NSLSCS, School of Mathematical Sciences,  Nanjing Normal University, Nanjing 210023, 
China (danlimath@163.com). The research of Dan Li was supported by China National Natural Science Foundation Grant (No. 12071227).}}
\and
Chunmei Wang \thanks{Department of Mathematics, University of Florida, Gainesville, FL 32611 (chunmei.wang@ufl.edu).  The research of Chunmei Wang was partially supported by National Science Foundation Grants DMS-2136380 and DMS-2206332.}
\and
Junping Wang\thanks{Division of Mathematical Sciences, National Science Foundation, Alexandria, VA 22314 (jwang@nsf.gov). The research of Junping Wang was supported by the NSF IR/D program, while working at National Science Foundation. However, any opinion, finding, and conclusions or recommendations expressed in this material are those of the author and do not necessarily reflect the views of the National Science Foundation.}
\and
Xiu Ye \thanks{Department of Mathematics, University of Arkansas at Little Rock, Little Rock, AR 72204, USA (xxye@ualr.edu).}}

\maketitle

\begin{abstract}
This article proposes and analyzes the generalized weak Galerkin ({\rm g}WG) finite element method for the second order elliptic problem. A generalized discrete weak gradient operator is introduced in the weak Galerkin framework so that the {\rm g}WG  methods would not only allow arbitrary combinations of piecewise polynomials defined in the interior and on the boundary of each local finite element, but also work on general polytopal partitions. Error estimates are established for the corresponding numerical functions in the energy norm and the usual $L^2$ norm. A series of numerical experiments are presented to demonstrate the performance of the newly proposed {\rm g}WG method.
\end{abstract}



\section{Introduction}
This paper will study the generalized weak Galerkin methods for the second order elliptic problems. For simplicity, we consider the second order model problem that seeks an unknown function $u$ satisfying
\begin{equation}\label{model-problem}
\begin{split}
-\nabla\cdot({\color{black}{a}}\nabla u)&=f,\quad \mbox{in}~~\O,\\
                      u&=g,\quad \mbox{on}~~\pa\O,
\end{split}
\end{equation}
where $\Omega$ is a bounded polytopal domain in $\mathbb{R}^d$$ (d=2,3)$ and the coefficient tensor ${\color{black}{a}}\in \mathbb{R}^{d\times d}$ is  symmetric and uniformly positive definite. 

The weak formulation of the model problem \eqref{model-problem} is as follows: Find $u\in H^1(\Omega)$ satisfying $u=g$ on $\partial\Omega$, such that
\begin{equation}\label{weakform}
({\color{black}{a}}\nabla u,\nabla v)=(f,v), \qquad\forall v\in H_0^1(\Omega),
\end{equation}
where $H_0^1(\Omega)=\{v\in H^1(\Omega): v=0\ \text{on}\ \partial\Omega\}$.

 Numerous numerical methods have been developed for solving the second order elliptic problems. The conforming finite element method is widely employed in scientific and engineering applications due to its simplicity and robustness. However, for certain model problems involving high-order partial differential equations, it is difficult to construct  the conforming finite element. To address this challenge, several numerical methods have been introduced, such as the discontinuous Galerkin method
\cite{ABCM_1999,ABCM_2001}, the hybrid discontinuous Galerkin method \cite{BGL_2009}, the mimetic finite differences method \cite{LMBB2011}, the hybrid-high order method \cite{PT_2018}, the virtual element method \cite{VBCMR_2013,VBMR_2016} and the weak Galerkin  finite element method \cite{WWL-2022-cur,MWWY2012,ellip_MC2014,ellip_reduced2015,ellip_WMGM2014,Wang-JSC-2018,ellip_XYZ2020}.

The weak Galerkin methods were first proposed in \cite{ellip_MC2014} for the second order elliptic problems. The most novel aspect of WG methods is the introduction of locally designed weak partial derivatives. This innovation allows WG methods to offer several advantages, including high flexibility in polynomial approximations and mesh generation.  WG methods have been widely applied in solving a diverse range of PDEs
\cite{cwwinterface,cwwdivcurl,lwwhyper,wcauchy2,wwtrans,wwcauchy,wwnondiv,wzcondif,WW-fp-2020}. Later on, the primal-dual weak Galerkin (PDWG) methods were proposed to simulate certain model problems that are challenging to solve using the traditional numerical methods, such as second-order elliptic equations in non-divergence form \cite{PDWGMC}, the Fokker-Planck equations \cite{DW2020,WW-fp-2018}, the elliptic Cauchy problems \cite{ellipCau_WW2020}, the first-order transport problems \cite{wwhyperbolic}, the div-curl systems with low-regularity solutions \cite{CWW-divcurl-2021,LW-divcurl-2020}. Recently, the PDWG methods have been extended to a more general $L^p$ setting by using the $L^p$-primal-dual weak Galerkin methods. The $L^p$-PDWG methods have been developed for div-curl systems \cite{CWWdiv}, second order elliptic equations in non-divergence form \cite{CWWsecond}, convection-diffusion equations \cite{CWW2023}, transport problems \cite{WWL-2023-tran}.



This paper aims to develop a generalized weak Galerkin method for the second order elliptic problem \eqref{model-problem}. 
 Different combinations of finite elements lead to different weak Galerkin methods in which a typical {\rm g}WG element is of the form $P_k(T)/P_j(\pa T)/[P_{\ell}(T)]^d$ where the weak function is discretized by the polynomial spaces $P_k(T)$ and $P_j(\pa T)$, and its generalized discrete weak gradient is approximated by the vector-valued space $[P_{\ell}(T)]^d$.  
 We have rigorously established the theory for the error estimates in a discrete norm  and the usual $L^2$ norm for the newly proposed {\rm g}WG methods. A series of numerical results have been demonstrated to verify the established theory. Compared with other existing results on the standard weak Galerkin methods, the {\rm g}WG methods can achieve a convergence rate in a  higher order  
for some combinations of weak finite elements. For example, for the $P_1(T)/P_0(\pa T)/[P_1(T)]^2$ element,   the standard WG method diverges;  while the {\rm g}WG method converges in an order  $ \mathcal {O}(h^2) $ in the $L^2$ norm.

The paper is organized as follows. In Section \ref{Section:WeakGradient}, we introduce the definition of a generalized discrete weak gradient. Section \ref{Section:WG} presents the generalized weak Galerkin scheme   for the model problem \eqref{model-problem}. Section \ref{Section:EQ} derives an error equation for the generalized weak Galerkin scheme.  Section \ref{Section:TRs}   presents some technical results. Section \ref{Section:EEs} is devoted to establishing some error estimates for the numerical approximation in a discrete norm and the usual $L^2$ norm. 
In Section \ref{Section:EE},  various numerical experiments are demonstrated.

Throughout this paper, we will follow the standard definitions for the Sobolev spaces and norms. Let $D$ be any open bounded domain with Lipschitz continuous boundary in $\mathbb{R}^d (d=2,3)$. Denote by $|\cdot|_{s, D}$ and $\|\cdot\|_{s,D}$ the seminorm and norm in the Sobolev space $H^s(D)$ for any integer $s\geq 0$, respectively. When $s=0$, the inner product and norm are denoted by $(\cdot, \cdot)_{D}$ and $\|\cdot\|_{D}$, respectively. When $D=\O$, the subscript $D$ shall be dropped in the corresponding inner product, seminorm and norm. We use the notation ``$\lesssim$'' to mean ``no greater than a generic positive constant independent of the meshsize or functions appearing in the inequalities''.

\section{Generalized Discrete Weak Gradient}\label{Section:WeakGradient}

The goal of this section is to define the generalized discrete weak gradient. To this end, let ${\cal T}_h$ be a finite element partition of $\O$ that satisfies the shape regular assumption as described as in \cite{ellip_MC2014}. Denote by $\E_h$ the set of all edges or flat faces in ${\cal T}_h$ and $\E_h^0=\E_h\setminus\pa\O$ the set of all interior edges or flat faces, respectively. Let $h_T$ be the diameter of $T\in{\cal T}_h$ and $h=\max_{T\in {\cal T}_h}h_T$ be the meshsize of the partition ${\cal T}_h$. For any given integer $r\geq0$, denote by $P_r(T)$  the set of polynomials defined on $T$ with degree no more than $r$.

Let $T\in{\cal T}_h$ be any polytopal element with boundary $\pa T$. By a weak function on $T$ we mean $v=\{v_0,v_b\}$ with $v_0\in L^2(T)$ and $v_b\in L^2(\pa T)$. The first component $v_0$ and the second component $v_b$ represent the values of $v$ in the interior and on the boundary of $T$, respectively. It should be pointed out that $v_b$ may not necessarily be the trace of $v_0$ on $\pa T$. 
Let $k\geq0$ and $j\geq0$ be two given integers. Let $V_{k,j}(T)$ be the local weak function space on each $T\in{\cal T}_h$ given by
$$
V_{k,j}(T)=\{v=\{v_0,v_b\}:v_k\in P_k(T),v_b\in P_j(e),~~e\subset \pa T\}.
$$ 

\begin{definition} (Generalized discrete weak gradient)
Let $\ell\geq0$ be a given integer. A generalized discrete weak gradient for any weak function $v\in V_{k,j}(T)$, denoted by $\nabla_{g,T}v$, is given by
\begin{equation}\label{weak gradient-1}
\begin{split}
\nabla_{g,T}v:=\nabla v_0+\delta_g v,
\end{split}
\end{equation}
where  $\delta_g v\in [P_\ell(T)]^d$ satisfies
\begin{equation}\label{weak gradient-2}
(\delta_gv,\boldsymbol{\psi})_T:=\langle v_b-Q_b v_0,\boldsymbol{\psi}\cdot\textbf{n}\rangle_{\pa T},\quad
    \forall\boldsymbol{\psi}\in[P_\ell(T)]^d,
\end{equation}
where $\textbf{n}$ is the unit outward normal direction to $\pa T$, and $Q_b$ is the usual $L^2$ projection operator onto $P_j(e)$.
\end{definition}

\section{Generalized Weak Galerkin Scheme}\label{Section:WG}

This section presents a generalized weak Galerkin scheme for the model problems \eqref{model-problem}. 
For simplicity of analysis, assume that  the coefficient ${\color{black}{a}}$ in \eqref{model-problem} is piecewise constant with respect to the finite element partition ${\cal T}_h$. The following result can be easily extended to variable coefficient tensor, provided that the tensor ${\color{black}{a}}$ is piecewise sufficiently smooth.

The global weak finite element space $V_h$ is obtained by patching   the local weak function space $V_{k,j}(T)$ over all the elements through a common value $v_b$ on the interior edges or faces $\mathcal{E}_h^0$; i.e.,
$$
V_h=\{v=\{v_0,v_b\}:v|_T\in V_{k,j}(T),~~T\in{\cal T}_h\}.
$$
Denote by $V_h^0$ the subspace of $V_h$ consisting of the weak functions with vanishing boundary value on $\pa\O$ given by
$$
V_h^0=\{v: v\in V_h,~v_b|_e=0,~~e\subset\pa\O\}.
$$

For simplicity of the notation, denote by $\nabla_g$ the generalized discrete weak gradient $\nabla_{g,T}$ computed by \eqref{weak gradient-1}-\eqref{weak gradient-2}; i.e.,
$$(\nabla_gv)|_T=\nabla_{g,T}(v|_T),~~~~~~~v\in V_h.$$

For any $w,v\in V_h$, we introduce the following bilinear forms; i.e.,
\begin{equation*}\label{stabilizer}
\begin{split}
&a(w,v)=\sum_{T\in{\cal T}_h}({\color{black}{a}}\nabla_{g}w,\nabla_{g}v)_T,\\
&s(w,v)=\sum_{T\in{\cal T}_h}\rho h_T^{\gamma}\langle Q_bw_0-w_b,Q_bv_0-v_b\rangle_{\pa T},
\end{split}
\end{equation*}
where $\rho>0$ and $\gamma\in \mathbb R.$
\begin{gWG}
A numerical approximation for the model problems \eqref{model-problem} based on the weak formulation \eqref{weakform} can be obtained by seeking $u_h=\{u_0,u_b\}\in V_h$ such that $u_b=Q_bg$ on $\pa\O$ satisfying
\begin{equation}\label{WG-scheme}
a(u_h,v)+s(u_h,v)=(f,v_0),\qquad \forall v\in V_h^0.
\end{equation}
\end{gWG}

\begin{lemma}\label{pre-error-equation}
The generalized weak Galerkin scheme \eqref{WG-scheme} has one and only one numerical approximation.
\end{lemma} 
\begin{proof}
It suffices to show that the homogeneous
 {\rm gWG} scheme \eqref{WG-scheme} has only the trivial solution.  To this end, we take $f=0$ and $g=0$. Let  $v=u_h\in V_h^0$ in \eqref{WG-scheme} gives $(a\nabla_gu_h,\nabla_gu_h)=0$ and $s(u_h,u_h)=0$. This leads to $\nabla_gu_h=0$ on each $T$ and $Q_bu_0=u_b$ on each $\pa T$. Using the generalized weak gradient \eqref{weak gradient-1}-\eqref{weak gradient-2} gives $\nabla u_0=0$ on each $T$ and further $u_0=const$ on each $T$. It follows from $Q_bu_0=u_b$ on each $\pa T$ and $u_b=0$ on $\pa\O$ that $u_0=0$ in $\O$ and $u_b=0$ on each $\pa T$. This completes the proof of the lemma.
\end{proof}


\section{Error Equations}\label{Section:EQ}

This section is devoted to deriving an error equation for the {\rm g}WG scheme \eqref{WG-scheme}. To this end, on each element $T\in{\cal T}_h$, denote by $Q_0$ the usual $L^2$ projection projector onto $P_k(T)$. For each $\phi\in H^1(T)$, let $Q_h\phi\in V_h$ be the $L^2$ projection such that on each element $T$,  we have
$$Q_h\phi=\{Q_0\phi,Q_b\phi\}.$$

Let $s=\min\{j,\ell\}$. Denote by $\mathbb{Q}_s$ the usual $L^2$ projection operator onto $[P_s(T)]^d$.

\begin{lemma} \label{pre-commu}
For any $\boldsymbol{\psi}_s\in [P_s(T)]^d$ and $\phi\in H^1(T)$, there holds
$$
(\nabla_gQ_h\phi,\boldsymbol{\psi}_s)_T=(\nabla\phi,\boldsymbol{\psi}_s)_T+(\phi-Q_0\phi,\nabla\cdot\boldsymbol{\psi}_s)_T.
$$
\end{lemma}
\begin{proof}
Using the definition of generalized discrete weak gradient \eqref{weak gradient-1}-\eqref{weak gradient-2}, $s=\min\{j,\ell\}$ and the usual integration by parts gives
\begin{equation*}
\begin{split}
 &(\nabla_gQ_h\phi,\boldsymbol{\psi}_s)_T\\
=&(\nabla Q_0\phi+\delta_gQ_h\phi,\boldsymbol{\psi}_s)_T\\
=&(\nabla Q_0\phi,\boldsymbol{\psi}_s)_T+\langle Q_b\phi-Q_b(Q_0\phi),\boldsymbol{\psi}_s\cdot\textbf{n}\rangle_{\pa T}\\
=&(\nabla Q_0\phi,\boldsymbol{\psi}_s)_T+\langle\phi-Q_0\phi,\boldsymbol{\psi}_s\cdot\textbf{n}\rangle_{\pa T}\\
=&(\nabla Q_0\phi,\boldsymbol{\psi}_s)_T+(\nabla\phi,\boldsymbol{\psi}_s)_T+(\phi,\nabla\cdot\boldsymbol{\psi}_s)_T
  -(Q_0\phi,\nabla\cdot\boldsymbol{\psi}_s)_T-(\nabla Q_0\phi,\boldsymbol{\psi}_s)_T\\
=&(\nabla\phi,\boldsymbol{\psi}_s)_T+(\phi-Q_0\phi,\nabla\cdot\boldsymbol{\psi}_s)_T.
\end{split}
\end{equation*}
This completes the proof of the lemma.
\end{proof}

Let $u_h\in V_h$   be the numerical solution of the {\rm g}WG scheme \eqref{WG-scheme} and  $u$ be the exact solution of the model problem  \eqref{model-problem}. Denote by $e_h$ the error function given by
\begin{equation}\label{error-function}
\begin{split}
e_h=Q_hu-u_h=\{e_0,e_b\}=\{Q_0u-u_0,Q_bu-u_b\}.
\end{split}
\end{equation}

\begin{lemma} \label{error-equation}
Let $e_h$ be the error function defined in \eqref{error-function}. Then, the following error equation holds true
\begin{eqnarray}\label{Error-equation-22}
\sum_{T\in{\cal T}_h}({\color{black}{a}}\nabla_ge_h,\nabla_gv)_{{ T}}+s(e_h,v)=\zeta_u(v),~~~~~\forall v\in V_h^0,
\end{eqnarray}
where $\zeta_u(v)$ is given by
\begin{equation}\label{error equation-remainder}
\begin{split}
\zeta_u(v)
=&s(Q_hu,v)+\sum_{T\in{\cal T}_h}(u-Q_0u,\nabla\cdot({\color{black}{a}}\mathbb{Q}_s\nabla_gv))_T\\
  & +\sum_{T\in{\cal T}_h}((\mathbb{Q}_s-I)({\color{black}{a}}\nabla u),\nabla v_0)_T
 +\sum_{T\in{\cal T}_h}\langle(I-\mathbb{Q}_s)({\color{black}{a}}\nabla u)\cdot\textbf{n},v_0-v_b\rangle_{\pa T}\\
 & +\sum_{T\in{\cal T}_h}({\color{black}{a}}\nabla_gQ_hu,(I-\mathbb{Q}_s)\nabla_gv)_T.
\end{split}
\end{equation}
\end{lemma}

\begin{proof}
From Lemma \ref{pre-commu} with $\boldsymbol{\psi}_s={\color{black}{a}}\mathbb{Q}_s\nabla_gv$ and $\phi=u$, one arrives at
\begin{eqnarray}\label{error-equation-1}
\begin{split}
(\nabla_gQ_hu,{\color{black}{a}}\mathbb{Q}_s\nabla_gv)_T
=(\nabla u,{\color{black}{a}}\mathbb{Q}_s\nabla_gv)_T+(u-Q_0u,\nabla\cdot({\color{black}{a}}\mathbb{Q}_s\nabla_gv))_T.
\end{split}
\end{eqnarray}

As to the first term on the right hand side of \eqref{error-equation-1}, using \eqref{weak gradient-1}-\eqref{weak gradient-2}, $s=\min\{j,\ell\}$, \eqref{WG-scheme}, and the usual integration by parts gives
\begin{equation}\label{error-equation-1-1}
\begin{split}
 &\sum_{T\in{\cal T}_h}(\nabla u,{\color{black}{a}}\mathbb{Q}_s\nabla_gv)_T\\
=&\sum_{T\in{\cal T}_h}(\mathbb{Q}_s({\color{black}{a}}\nabla u),\nabla_gv)_T\\
=&\sum_{T\in{\cal T}_h}(\mathbb{Q}_s({\color{black}{a}}\nabla u),\nabla v_0)_T
  +\langle v_b-Q_bv_0,\mathbb{Q}_s({\color{black}{a}}\nabla u)\cdot\textbf{n}\rangle_{\pa T}\\
=&\sum_{T\in{\cal T}_h}((\mathbb{Q}_s-I)({\color{black}{a}}\nabla u),\nabla v_0)_T-(\nabla\cdot({\color{black}{a}}\nabla u),v_0)_T
  +\langle {\color{black}{a}}\nabla u\cdot\textbf{n},v_0\rangle_{\pa T}\\
 &+\langle v_b-v_0,\mathbb{Q}_s({\color{black}{a}}\nabla u)\cdot\textbf{n}\rangle_{\pa T}\\
=&\sum_{T\in{\cal T}_h}((\mathbb{Q}_s-I)({\color{black}{a}}\nabla u),\nabla v_0)_T+(f,v_0)
  +\sum_{T\in{\cal T}_h}\langle(I-\mathbb{Q}_s)({\color{black}{a}}\nabla u)\cdot\textbf{n},v_0-v_b\rangle_{\pa T}\\
=&\sum_{T\in{\cal T}_h}((\mathbb{Q}_s-I)({\color{black}{a}}\nabla u),\nabla v_0)_T+({\color{black}{a}}\nabla_gu_h,\nabla_gv)_{T}+s(u_h,v)\\
 &+\sum_{T\in{\cal T}_h}\langle(I-\mathbb{Q}_s)({\color{black}{a}}\nabla u)\cdot\textbf{n},v_0-v_b\rangle_{\pa T}\\
=&\sum_{T\in{\cal T}_h}((\mathbb{Q}_s-I)({\color{black}{a}}\nabla u),\nabla v_0)_T-({\color{black}{a}}\nabla_ge_h,\nabla_gv)_{T}
  +({\color{black}{a}}\nabla_gQ_hu,\nabla_gv)_{T}\\
 &-s(e_h,v)+s(Q_hu,v)+\sum_{T\in{\cal T}_h}\langle(I-\mathbb{Q}_s)({\color{black}{a}}\nabla u)\cdot\textbf{n},v_0-v_b\rangle_{\pa T},
\end{split}
\end{equation}
where we have also used  the first equation in \eqref{model-problem}, the fact $\sum_{T\in{\cal T}_h}\langle {\color{black}{a}}\nabla u\cdot\textbf{n},v_b\rangle_{\pa T}=0$ since $v_b=0$ on $\partial\Omega$.

Finally, substituting \eqref{error-equation-1-1} into \eqref{error-equation-1} gives rise to \eqref{Error-equation-22}. This completes the proof of the lemma.
\end{proof}

\section{Technical Results}\label{Section:TRs}
  Some technical results will be discussed in this section.

Let ${\cal T}_h$ be a finite element partition of $\O$ that is shape regular as described in \cite{ellip_MC2014}. For any $T\in{\cal T}_h$ and $\phi\in H^1(T)$, the   trace inequality holds true \cite{ellip_MC2014}; i.e.,
\begin{equation}\label{trace-inequality}
\|\phi\|_{\pa T}^2\lesssim h_T^{-1}\|\phi\|_T^2+h_T\|\nabla\phi\|_T^2.
\end{equation}
If $\phi$ is a polynomial on $T\in{\cal T}_h$, from the inverse inequality, there holds \cite{ellip_MC2014}
\begin{equation}\label{inverse-inequality}
\|\phi\|_{\pa T}^2\lesssim h_T^{-1}\|\phi\|_T^2.
\end{equation}

\begin{lemma} \label{error estimate for projection}
Let ${\cal T}_h$ be a finite element partition of $\O$ that is shape regular as described in \cite{ellip_MC2014}. For any $\phi\in H^{k+1}(\O)$ and $\varphi\in H^{s+2}(\O)$, there holds
\begin{equation}\label{e1}
\begin{split}
&\sum_{T\in{\cal T}_h}\|\phi-Q_{0}\phi\|_T^2+\sum_{T\in{\cal T}_h}h_T^2\|\nabla(\phi-Q_{0}\phi)\|_T^2\lesssim h^{2(k+1)}\|\phi\|_{k+1}^2,
 \end{split}
\end{equation}
\begin{equation}\label{e2}
\begin{split}
\sum_{T\in{\cal T}_h}\|\nabla\varphi-\mathbb{Q}_s\nabla\varphi\|_T^2
 +\sum_{T\in{\cal T}_h}h_T^2\|\nabla(\nabla\varphi-\mathbb{Q}_s\nabla\varphi)\|_T^2
\lesssim h^{2(s+1)}\|\varphi\|_{s+2}^2.
 \end{split}
\end{equation}
\end{lemma}

\begin{lemma} \label{gradient property}
For any $\phi\in H^{k+1}(T)$, there holds
$$
\|\delta_gQ_h\phi\|_T\lesssim h_T^k\|\phi\|_{k+1,T}.
$$
\end{lemma}
\begin{proof}
It follows from \eqref{weak gradient-2}, the Cauchy-Schwarz inequality, the trace inequality \eqref{inverse-inequality} and \eqref{e1} that
\begin{equation*}
\begin{split}
\|\delta_gQ_h\phi\|_T
=&\sup_{\boldsymbol{\psi}\in[P_\ell(T)]^d}\frac{(\delta_gQ_h\phi,\boldsymbol{\psi})_T}{\|\boldsymbol{\psi}\|_T}\\
=&\sup_{\boldsymbol{\psi}\in[P_\ell(T)]^d}\frac{\langle Q_b\phi-Q_b(Q_0\phi),\boldsymbol{\psi}\cdot\textbf{n}\rangle_{\pa T}}
  {\|\boldsymbol{\psi}\|_T}\\
\lesssim&\sup_{\boldsymbol{\psi}\in[P_\ell(T)]^d}\frac{\|\phi-Q_0\phi\|_{\pa T}\|\boldsymbol{\psi}\|_{\pa T}}
  {\|\boldsymbol{\psi}\|_T}\\
\lesssim& \sup_{\boldsymbol{\psi}\in[P_\ell(T)]^d}\frac{h_{T}^{-1}\|\phi-Q_0\phi\|_{T}\|\boldsymbol{\psi}\|_{T}}
  {\|\boldsymbol{\psi}\|_T}\\
\lesssim & h_T^k\|\phi\|_{k+1,T}.
\end{split}
\end{equation*}
This completes the proof of the lemma.
\end{proof}

For any $v\in V_h$, the {\rm g}WG scheme \eqref{WG-scheme} induces a seminorm given by
\begin{equation}\label{three-bar}
\3barv\3bar^2=\sum_{T\in {\cal T}_h}({\color{black}{a}}\nabla_gv,\nabla_gv)_{T}+s(v,v).
\end{equation}
It is easy to verify $\3bar\cdot\3bar$ is a norm in $V_h^0$. 
\begin{lemma}\label{gradient-tribar}
For any $v\in V_h$, there holds 
$$
\Big(\sum_{T\in{\cal T}_h}\|\nabla v_0\|_T^2\Big)^{\frac{1}{2}}\lesssim(1+h^{\frac{-\gamma-1}{2}})\3barv\3bar.
$$
\end{lemma}
\begin{proof}
From the generalized discrete weak gradient \eqref{weak gradient-1} and \eqref{three-bar}, one arrives at
\begin{equation}\label{gradient-tribar-19:36}
\begin{split}
\Big(\sum_{T\in{\cal T}_h}\|\nabla v_0\|_T^2\Big)^{\frac{1}{2}}
=&\Big(\sum_{T\in{\cal T}_h}\|\nabla_gv-\delta_gv\|_T^2\Big)^{\frac{1}{2}}\\
\lesssim&\3barv\3bar+\Big(\sum_{T\in{\cal T}_h}\|\delta_gv\|_T^2\Big)^{\frac{1}{2}}.
\end{split}
\end{equation}
We use \eqref{weak gradient-2}, the Cauchy-Schwarz inequality and the trace inequality \eqref{inverse-inequality} to obtain
\begin{equation*} 
\begin{split}
\|\delta_gv\|_T
=&\sup_{\boldsymbol{\psi}\in[P_\ell(T)]^d}\frac{(\delta_gv,\boldsymbol{\psi})_T}{\|\boldsymbol{\psi}\|_T}\\
=&\sup_{\boldsymbol{\psi}\in[P_\ell(T)]^d}\frac{\langle v_b-Q_bv_0,\boldsymbol{\psi}\cdot\textbf{n}\rangle_{\pa T}}{\|\boldsymbol{\psi}\|_T}\\
\lesssim&\sup_{\boldsymbol{\psi}\in[P_\ell(T)]^d}\frac{\|v_b-Q_bv_0\|_{\pa T}\|\boldsymbol{\psi}\|_{\pa T}}{\|\boldsymbol{\psi}\|_T}\\
\lesssim&h_T^{ -\frac{1}{2} }  \|v_b-Q_bv_0\|_{\pa T}.
  \end{split}
\end{equation*}
This gives 
\begin{equation}\label{gradient-tribar-19:37}
    \Big(\sum_{T\in{\cal T}_h}\|\delta_gv\|_T^2\Big)^{\frac{1}{2}} \lesssim h^{\frac{-\gamma-1}{2}}\3bar v\3bar.
\end{equation}
Substituting \eqref{gradient-tribar-19:37}  into \eqref{gradient-tribar-19:36}   completes the proof of the lemma.
\end{proof}

\begin{lemma}\label{energy-estimate-pre}
Recall that $s=\min\{j,\ell\}$. For any $\varphi\in H^{k+1}(\O)\cap H^{s+2}(\O)$ and $v\in V_{h}$, there holds
\begin{equation}\label{technique-estimate-1}
|s(Q_{h}\varphi,v)|\lesssim h^{\frac{2k+1+\gamma}{2}}\|\varphi\|_{k+1}\3barv\3bar,
\end{equation}
\begin{equation}\label{technique-estimate-2}
\begin{split}
&|\sum_{T\in{\cal T}_h}(\varphi-Q_0\varphi,\nabla\cdot({\color{black}{a}}\mathbb{Q}_s\nabla_gv))_T|\\
\lesssim&\begin{cases} 0,  \qquad\mbox{if}~s\leq1,~ k\geq0,~or~s>1,~k\geq s-1,\\
h^k\|\varphi\|_{k+1}\3barv\3bar, \quad\mbox{if}~s>1,~ k<s-1,\end{cases}
\end{split}
\end{equation}

 \begin{equation}\label{technique-estimate-3}
\begin{split}
&|\sum_{T\in{\cal T}_h}((\mathbb{Q}_s-I)({\color{black}{a}}\nabla\varphi),\nabla v_0)_T|\\
\lesssim&\begin{cases} 0,  \qquad\qquad\mbox{if}~k\leq1,~or~s\geq k-1,\\
(h^{\frac{2s+1-\gamma}{2}}+h^{s+1})\|\varphi\|_{s+2}\3barv\3bar, \quad\mbox{otherwise},\end{cases}
\end{split}
\end{equation}

\begin{equation}\label{technique-estimate-4}
\begin{split}
&|\sum_{T\in{\cal T}_h}\langle(I-\mathbb{Q}_s)({\color{black}{a}}\nabla\varphi)\cdot\textbf{n},v_0-v_b\rangle_{\pa T}|
\\
\lesssim&\begin{cases} h^{\frac{2s+1-\gamma}{2}}\|\varphi\|_{s+2}\3barv\3bar, \qquad\qquad\qquad\qquad \mbox{if}~ k=0,\\
(h^{\frac{2s+1-\gamma}{2}}+h^{s+1})\|\varphi\|_{s+2}\3barv\3bar, \qquad\qquad\mbox{otherwise},\end{cases}
\end{split}
\end{equation}


\begin{equation}\label{technique-estimate-5}
\begin{split}
&|\sum_{T\in{\cal T}_h}({\color{black}{a}}\nabla_gQ_h\varphi,(I-\mathbb{Q}_s)\nabla_gv)_T|\\
\lesssim& \begin{cases}
0, \qquad\qquad\qquad\qquad\qquad\qquad\qquad\qquad~~~ \mbox{if}~ s\geq\max\{k-1,\ell\},\\
\|\varphi\|_{1}\3barv\3bar,\qquad\qquad\qquad\qquad~~~ \mbox{if}~k=0,~ s<\max\{k-1,\ell\},\\
(h^k\|\varphi\|_{k+1}+h^{s+1}\|\varphi\|_{s+2})\3barv\3bar, ~~~\qquad\qquad\mbox{otherwise}.\end{cases} 
\end{split}
\end{equation}

\end{lemma}

\begin{proof}
As to the first inequality \eqref{technique-estimate-1}, by using the Cauchy-Schwarz inequality, \eqref{trace-inequality} and Lemma \ref{error estimate for projection}, there holds
\begin{equation*}
\begin{split}
|s(Q_{h}\varphi,v)|
=&|\sum_{T\in {\cal T}_h} \rho h_T^{\gamma}\langle Q_b(Q_0\varphi)-Q_b\varphi,Q_bv_0-v_b\rangle_{\pa T}|\\
             \lesssim&\Big(\sum_{T\in {\cal T}_h}\rho h_T^{\gamma}\|Q_0\varphi-\varphi\|_{\pa T}^2\Big)^{\frac{1}{2}}\Big(\sum_{T\in {\cal T}_h} \rho h_T^{\gamma}\|Q_bv_0-v_b\|_{\pa T}^2\Big)^{\frac{1}{2}}\\
             \lesssim&\Big(\sum_{T\in {\cal T}_h}h_T^{\gamma}(h_T^{-1}\|Q_0\varphi-\varphi\|_{T}^2+h_T\|\nabla(Q_0\varphi-\varphi)\|_{T}^2)\Big)^{\frac{1}{2}}\3barv\3bar\\
             \lesssim&h^{\frac{2k+1+\gamma}{2}}\|\varphi\|_{k+1}\3barv\3bar.
\end{split}
\end{equation*}

 To derive \eqref{technique-estimate-2},  for the case of $s\leq1$ as well as the case of $s>1$ and $k\geq s-1$, we have from the definition of $L^2$ projection operator $Q_0$ that
$$|\sum_{T\in {\cal T}_h}(\varphi-Q_0\varphi,\nabla\cdot({\color{black}{a}}\mathbb{Q}_{s}\nabla_gv))_T|=0.$$
Otherwise, the case of $s>1$ and $k<s-1$, we use the Cauchy-Schwarz inequality, the inverse inequality, Lemma \ref{error estimate for projection} to obtain
\begin{equation*}
\begin{split}
&|\sum_{T\in {\cal T}_h}(\varphi-Q_0\varphi,\nabla\cdot({\color{black}{a}}\mathbb{Q}_{s}\nabla_gv))_T|\\
\lesssim&\Big(\sum_{T\in {\cal T}_h}\|Q_0\varphi-\varphi\|_{T}^2\Big)^{\frac{1}{2}}\Big(\sum_{T\in {\cal
           T}_h}\|\nabla\cdot({\color{black}{a}}\mathbb{Q}_{s}\nabla_gv)\|_{T}^2\Big)^{\frac{1}{2}}\\
\lesssim& h^{k+1}\|\varphi\|_{k+1}\Big(\sum_{T\in {\cal T}_h}h_T^{-2}\|\nabla_gv\|_T^2\Big)^{\frac{1}{2}}\\
\lesssim& h^k\|\varphi\|_{k+1}\3barv\3bar.
\end{split}
\end{equation*}

To analyze the inequality \eqref{technique-estimate-3}, for the case of $k\leq1$ and the case of $s\geq k-1$, there holds
$$|\sum_{T\in{\cal T}_h}((\mathbb{Q}_s-I)({\color{black}{a}}\nabla\varphi),\nabla v_0)_T|=0.$$
For the case of $k>1$ and the case of $s<k-1$, using the Cauchy-Schwarz inequality, \eqref{e2}, and Lemma \ref{gradient-tribar} gives
\begin{equation*}
\begin{split}
|\sum_{T\in{\cal T}_h}((\mathbb{Q}_s-I)({\color{black}{a}}\nabla\varphi),\nabla v_0)_T|
\lesssim&\Big(\sum_{T\in{\cal T}_h}\|(\mathbb{Q}_s-I)({\color{black}{a}}\nabla\varphi)\|_{T}^2\Big)^{\frac{1}{2}}
        \Big(\sum_{T\in{\cal T}_h}\|\nabla v_0\|_{T}^2\Big)^{\frac{1}{2}}\\
\lesssim&h^{s+1}(1+h^{\frac{-\gamma-1}{2}})\|\varphi\|_{s+2}\3barv\3bar.
\end{split}
\end{equation*}

As to  \eqref{technique-estimate-4}, for the case of $k=0$, note that $v_0=Q_bv_0$, using the triangle inequality \eqref{trace-inequality}, and the Cauchy-Schwarz inequality gives
\begin{equation*}
\begin{split}
 &|\sum_{T\in{\cal T}_h}\langle(I-\mathbb{Q}_s)({\color{black}{a}}\nabla\varphi)\cdot\textbf{n},v_0-v_b\rangle_{\pa T}|\\
\lesssim&|\sum_{T\in {\cal T}_h}\langle(I-\mathbb{Q}_{s}){\color{black}{a}}\nabla \varphi\cdot\textbf{n},Q_bv_0-v_b\rangle_{\pa T}| \\
\lesssim&\Big(\sum_{T\in {\cal T}_h}\rho^{-1}h_T^{-\gamma}\|(I-\mathbb{Q}_{s}){\color{black}{a}}\nabla \varphi\|_{\pa T}^2\Big)^{\frac{1}{2}}\Big(\sum_{T\in {\cal T}_h} \rho
     h_T^{\gamma}\|Q_bv_0-v_b\|_{\pa T}^2\Big)^{\frac{1}{2}}\\
\lesssim&\Big(\sum_{T\in {\cal T}_h}\rho^{-1}h_T^{-\gamma}(h_T^{-1}\|(I-\mathbb{Q}_{s}){\color{black}{a}}\nabla \varphi\|_{T}^2+h_T\|\nabla((I-\mathbb{Q}_{s}){\color{black}{a}}\nabla
     \varphi)\|_{T}^2)\Big)^{\frac{1}{2}}\3barv\3bar\\
\lesssim&h^{\frac{2s+1-\gamma}{2}}\|\varphi\|_{s+2}\3barv\3bar.
\end{split}
\end{equation*}
For the case of $k>0$, we use the Cauchy-Schwarz inequality, the trace inequalities \eqref{trace-inequality}-\eqref{inverse-inequality}, \eqref{e2},  and Lemma \ref{gradient-tribar} to obtain
\begin{equation*}
\begin{split}
 &|\sum_{T\in{\cal T}_h}\langle(I-\mathbb{Q}_s)({\color{black}{a}}\nabla\varphi)\cdot\textbf{n},v_0-v_b\rangle_{\pa T}|\\
\lesssim&\Big(\sum_{T\in{\cal T}_h}\|(I-\mathbb{Q}_s)({\color{black}{a}}\nabla\varphi)\|_{\pa T}^2\Big)^{\frac{1}{2}}
     \Big(\sum_{T\in{\cal T}_h}\|v_0-v_b\|_{\pa T}^2\Big)^{\frac{1}{2}}\\
\lesssim&\Big(\sum_{T\in{\cal T}_h}h_T^{-1}\|(I-\mathbb{Q}_s){\color{black}{a}}\nabla\varphi\|_{T}^2
    +\sum_{T\in{\cal T}_h}h_T\|\nabla((I-\mathbb{Q}_s){\color{black}{a}}\nabla\varphi)\|_{T}^2\Big)^{\frac{1}{2}}\\
     &\cdot\Big(\sum_{T\in{\cal T}_h}\|v_0-Q_bv_0\|_{\pa T}^2+\sum_{T\in{\cal T}_h}\|Q_bv_0-v_b\|_{\pa T}^2\Big)^{\frac{1}{2}}\\
\lesssim&\Big(\sum_{T\in{\cal T}_h}h_T^{2s+2-1}\|\varphi\|_{s+2}^2\Big)^{\frac{1}{2}}
     \Big(\sum_{T\in{\cal T}_h}h_T^{2}|v_0|_{1,\pa T}^2+h^{-\gamma}\3barv\3bar^2\Big)^{\frac{1}{2}}\\
\lesssim&h^{\frac{2s+1}{2}}\|\varphi\|_{s+2}\Big(\sum_{T\in{\cal T}_h}h_T^2h_T^{-1}|\nabla v_0|_{T}^2+h^{-\gamma}\3barv\3bar^2\Big)^{\frac{1}{2}}\\
\lesssim&(h^{\frac{2s+1-\gamma}{2}}+h^{s+1})\|\varphi\|_{s+2}\3barv\3bar.
\end{split}
\end{equation*}

As to \eqref{technique-estimate-5}, for the case of $s\geq\max\{k-1,\ell\}$, it follows from the definition of the generalized discrete weak gradient \eqref{weak gradient-1}-\eqref{weak gradient-2}, and  the definition of $\mathbb{Q}_s$ that 
$$|\sum_{T\in{\cal T}_h}({\color{black}{a}}\nabla_gQ_h\varphi,(I-\mathbb{Q}_s)\nabla_gv)_T|=0.
$$

For the case of  $s<\max\{k-1,\ell\}$,  we apply \eqref{weak gradient-1}, Lemma \ref{pre-commu} with $\boldsymbol{\psi}_s={\color{black}{a}}\mathbb{Q}_s\nabla_gv$, the Cauchy-Schwarz inequality, Lemmas \ref{error estimate for projection}-\ref{gradient property},  and the inverse inequality to obtain
\begin{equation*}
\begin{split}
 &|\sum_{T\in{\cal T}_h}({\color{black}{a}}\nabla_gQ_h\varphi,(I-\mathbb{Q}_s)\nabla_gv)_T|\\
=&|\sum_{T\in{\cal T}_h}(\nabla_gQ_h\varphi,{\color{black}{a}}\nabla_gv)_T
  -(\nabla_gQ_h\varphi,{\color{black}{a}}\mathbb{Q}_s\nabla_gv)_T|\\
=&|\sum_{T\in{\cal T}_h}(\nabla Q_0\varphi+\delta_gQ_h\varphi,{\color{black}{a}}\nabla_gv)_T-(\nabla\varphi,{\color{black}{a}}\mathbb{Q}_s\nabla_gv)_T
  -(\varphi-Q_0\varphi,\nabla\cdot({\color{black}{a}}\mathbb{Q}_s\nabla_gv))_T|\\
=&|\sum_{T\in{\cal T}_h}(\nabla Q_0\varphi-\nabla\varphi,{\color{black}{a}}\nabla_gv)_T
  +(\nabla\varphi-\mathbb{Q}_s\nabla\varphi,{\color{black}{a}}\nabla_gv)_T+(\delta_gQ_h\phi,{\color{black}{a}}\nabla_gv)_T\\
 &-(\varphi-Q_0\varphi,\nabla\cdot({\color{black}{a}}\mathbb{Q}_s\nabla_gv))_T|\\
\lesssim&\Big(\sum_{T\in{\cal T}_h}\|\nabla Q_0\varphi-\nabla\varphi\|_{T}^2\Big)^{\frac{1}{2}}\3barv\3bar+\Big(\sum_{T\in{\cal T}_h}\|\nabla\varphi-\mathbb{Q}_s\nabla\varphi\|_{T}^2\Big)^{\frac{1}{2}}\3barv\3bar\\
&+\Big(\sum_{T\in{\cal T}_h}\|\delta_gQ_h\varphi\|_{T}^2\Big)^{\frac{1}{2}}\3barv\3bar+\Big(\sum_{T\in{\cal T}_h}\|\varphi-Q_0\varphi\|_{T}^2\Big)^{\frac{1}{2}}\Big(\sum_{T\in{\cal T}_h}\|\nabla({\color{black}{a}}\mathbb{Q}_s\nabla_gv)\|_{T}^2\Big)^{\frac{1}{2}}\\
\lesssim&\begin{cases} \Big(\sum_{T\in{\cal T}_h}\|\varphi\|_{1,T}^2\Big)^{\frac{1}{2}}\3barv\3bar
    +\Big(\sum_{T\in{\cal T}_h}h_T^2\|\varphi\|_{1,T}^2\Big)^{\frac{1}{2}}\Big(\sum_{T\in{\cal T}_h}h_T^{-2}\|\nabla_g\varphi\|_{T}^2\Big)^{\frac{1}{2}}, \\
    ~~~~~~ if~ k=0,~s<\max\{k-1,\ell\},&\\
 (h^k\|\varphi\|_{k+1}+h^{s+1}\|\varphi\|_{s+2}+h^k\|\varphi\|_{k+1}+h^{k+1}h^{-1}\|\varphi\|_{k+1})\3barv\3bar, \\
~~~~~~~~  if~ k>0,~s<\max\{k-1,\ell\},&
\end{cases}\\
\lesssim&\begin{cases} \|\varphi\|_{1}\3barv\3bar,~~~~~~~~~~~~~~~~~~~~~~~~~~~~~~if~ k=0,~s<\max\{k-1,\ell\}, &\\ (h^k\|\varphi\|_{k+1}+h^{s+1}\|\varphi\|_{s+2})\3barv\3bar,
~~~~~~~~  if~ k>0,~s<\max\{k-1,\ell\}.&
\end{cases}
\end{split}
\end{equation*}
This completes the proof of the lemma.
\end{proof}

\begin{lemma}\label{L2-estimate-pre}
Recall $s=\min\{j,\ell\}$. { Let $\Phi\in H^2(\O)$ and $\Phi=0$ on $\pa\O$. For any $u\in H^{k+1}(\O)\cap H^{s+2}(\O)$,} there holds

\begin{equation}\label{L2-estimate-1}
\begin{split} |s(Q_{h}u,Q_{h}\Phi)|
&\lesssim\begin{cases} h^{\gamma+1}{ \|u\|_{1}}\|\Phi\|_2,& \mbox{if}~ k=0,\\
h^{k+\gamma+2}{ \|u\|_{k+1}}\|\Phi\|_2,&  \mbox{if}~ k>0,
         \end{cases}\\
\end{split}
\end{equation}

\begin{equation}\label{L2-estimate-2}
\begin{split} &|\sum_{T\in{\cal T}_h}(u-Q_0u,\nabla\cdot({\color{black}{a}}\mathbb{Q}_s\nabla_gQ_h\Phi))_T|\\
\lesssim&\begin{cases} 0,& ~\mbox{if}~s\leq1,~ k\geq0,~or~s>1,~k\geq s-1,\\
{ \|u\|_1}\|\Phi\|_2,& ~\mbox{if}~k=0,~ s>1,\\
h^{k+1}{ \|u\|_{k+1}}\|\Phi\|_2,&  \mbox{otherwise},
         \end{cases}
\end{split}
\end{equation}

\begin{equation}\label{L2-estimate-3}
\begin{split} &|\sum_{T\in{\cal T}_h}((\mathbb{Q}_s-I)({\color{black}{a}}\nabla u),\nabla Q_0\Phi)_T|\\
\lesssim&\begin{cases} 0,& \mbox{if}~k\leq1,or~s\geq k-1,\\
h^{s+2}{ \|u\|_{s+2}}\|\Phi\|_2,&  \mbox{otherwise},
         \end{cases}
\end{split}
\end{equation}

\begin{equation}\label{L2-estimate-4}
\begin{split} |\sum_{T\in {\cal T}_h}\langle(I-\mathbb{Q}_{s})({\color{black}{a}}\nabla u)\cdot\textbf{n},Q_0\Phi-Q_b\Phi\rangle_{\pa T}|
&\lesssim\begin{cases} h^{s+1}{ \|u\|_{s+2}}\|\Phi\|_2,& \mbox{if}~k=0,\\
h^{s+2}{ \|u\|_{s+2}}\|\Phi\|_2,&  \mbox{if}~k>0,
         \end{cases}
\end{split}
\end{equation}

\begin{equation}\label{L2-estimate-5}
\begin{split} &|\sum_{T\in{\cal T}_h}({\color{black}{a}}\nabla_gQ_hu,(I-\mathbb{Q}_s)\nabla_gQ_h\Phi)_T|\\
\lesssim&\begin{cases} 0,& \mbox{if}~s\geq\max\{k-1,\ell\},\\
\|u\|_1\|\Phi\|_2,& \mbox{if}~k=0,~s<\max\{k-1,\ell\},\\
{ (h^{k+1}\|u\|_{k+1}+h^{s+2}\|u\|_{s+2})\|\Phi\|_2,}&  \mbox{if}~k>0,~s<\max\{k-1,\ell\}.
         \end{cases}
\end{split}
\end{equation}
\end{lemma}

\begin{proof}
As to  \eqref{L2-estimate-1}, using the Cauchy-Schwarz inequality, the trace inequality \eqref{trace-inequality}, and \eqref{e1} gives
\begin{equation*}\label{error-equation-4}
\begin{split}
|s(Q_{h}u,Q_{h}\Phi)|
=&|\sum_{T\in {\cal T}_h}\rho h_T^{\gamma}\langle Q_b(Q_0u)-Q_bu,Q_b(Q_0\Phi)-Q_b\Phi\rangle_{\pa T}|\\
\lesssim&\Big(\sum_{T\in {\cal T}_h}\rho h_T^{\gamma}\|Q_0u-u\|_{\pa T}^2\Big)^{\frac{1}{2}}
\Big(\sum_{T\in {\cal T}_h}\rho h_T^{\gamma}\|Q_0\Phi-\Phi\|_{\pa T}^2\Big)^{\frac{1}{2}}\\
\lesssim&\Big(\sum_{T\in {\cal T}_h}h_T^{\gamma}(h_T^{-1}\|Q_0u-u\|_{T}^2+h_T\|\nabla(Q_0u-u)\|_{T}^2)\Big)^{\frac{1}{2}}\\
&\cdot\Big(\sum_{T\in {\cal T}_h}h_T^{\gamma}(h_T^{-1}\|Q_0\Phi-\Phi\|_{T}^2+h_T\|\nabla(Q_0\Phi-\Phi)\|_{T}^2)\Big)^{\frac{1}{2}}\\
\lesssim&
\left\{
\begin{array}{lr}
h^{\frac{\gamma+1}{2}}\|u\|_1h^{\frac{\gamma+1}{2}}\|\Phi\|_1,\quad if~k=0,\\
h^{\frac{2k+1+\gamma}{2}}h^{\frac{3+\gamma}{2}}\|u\|_{k+1}\|\Phi\|_2,\;~\qquad if~k>0,
\end{array}
\right.\\
\lesssim&
\left\{
\begin{array}{lr}
h^{\gamma+1}{ \|u\|_{1}}\|\Phi\|_2,\quad if~k=0,\\
h^{k+\gamma+2}{ \|u\|_{k+1}}\|\Phi\|_2,\;~\qquad if~k>0.
\end{array}
\right.
\end{split}
\end{equation*}

As to \eqref{L2-estimate-2}, for the case of $s\leq1$ and the case of $s>1$ and $k\geq s-1$, using  the definition of $Q_0$ gives
\begin{equation*}\label{error-equation-52}
\begin{split}
|\sum_{T\in{\cal T}_h}(u-Q_0u,\nabla\cdot({\color{black}{a}}\mathbb{Q}_s\nabla_gQ_h\Phi))_T|=0.
\end{split}
\end{equation*}
For the case of $s>1$ and $k<s-1$, using the Cauchy-Schwarz inequality, \eqref{e1}, the generalized weak gradient \eqref{weak gradient-1}, the inverse inequality and Lemma \ref{gradient property}, yields
\begin{equation*}
\begin{split}
&|\sum_{T\in{\cal T}_h}(u-Q_0u,\nabla\cdot({\color{black}{a}}\mathbb{Q}_s\nabla_gQ_h\Phi))_T|\\
\lesssim&\Big(\sum_{T\in{\cal T}_h}\|u-Q_0u\|_T^2\Big)^{\frac{1}{2}}
      \Big(\sum_{T\in{\cal T}_h}|\nabla_gQ_h\Phi|_{1,T}^2\Big)^{\frac{1}{2}}\\
\lesssim&h^{k+1}\Big(\sum_{T\in{\cal T}_h}|\nabla Q_0\Phi+\delta_gQ_h\Phi|_{1,T}^2\Big)^{\frac{1}{2}}\|u\|_{k+1}\\
\lesssim&h^{k+1}\Big(\sum_{T\in{\cal T}_h}|\nabla Q_0\Phi|_{1,T}^2+h_T^{-2}\|\delta_gQ_h\Phi\|_{T}^2\Big)^{\frac{1}{2}}\|u\|_{k+1}\\
\lesssim&
\left\{
\begin{array}{lr}
h\Big(\sum_{T\in{\cal T}_h}h_T^{-2}\|\Phi\|_{1,T}^2\Big)^{\frac{1}{2}}\|u\|_1,\quad if~k=0,\\
h^{k+1}\Big(\|\Phi\|_{2}^2+\sum_{T\in{\cal T}_h}h_T^{-2}h_T^2\|\Phi\|_{2,T}^2\Big)^{\frac{1}{2}}\|u\|_{k+1},\;~\qquad if~k>0,
\end{array}
\right.\\
\lesssim&
\left\{
\begin{array}{lr}
{ \|u\|_1}\|\Phi\|_2,\quad if~k=0,\\
h^{k+1}{ \|u\|_{k+1}}\|\Phi\|_2,\;~\qquad if~k>0.
\end{array}
\right.
\end{split}
\end{equation*}

As to \eqref{L2-estimate-3}, for the case of $k\leq1$ and the case of $s\geq k-1$, one arrives at
\begin{equation*}\label{error-equation-60}
\begin{split}
|\sum_{T\in{\cal T}_h}((\mathbb{Q}_s-I)({\color{black}{a}}\nabla u),\nabla Q_0\Phi)_T|=0.
\end{split}
\end{equation*}
For the case of $k>1$ and $s<k-1$, the definition of $\mathbb{Q}_s$, the Cauchy-Schwarz inequality and \eqref{e2} are used to obtain
\begin{equation*}\label{error-equation-5}
\begin{split}
 &|\sum_{T\in{\cal T}_h}((\mathbb{Q}_s-I)({\color{black}{a}}\nabla u),\nabla Q_0\Phi)_T|\\
=&|\sum_{T\in{\cal T}_h}((\mathbb{Q}_s-I)({\color{black}{a}}\nabla u),(I-\mathbb{Q}_s)\nabla Q_0\Phi)_T|\\
\lesssim&\Big(\sum_{T\in{\cal T}_h}\|(\mathbb{Q}_s-I){\color{black}{a}}\nabla u\|_T^2\Big)^{\frac{1}{2}}
     \Big(\sum_{T\in{\cal T}_h}\|(I-\mathbb{Q}_s)\nabla Q_0\Phi\|_T^2\Big)^{\frac{1}{2}}\\
\lesssim&h^{s+1}\|u\|_{s+2}\cdot h\|Q_0\Phi\|_2\\
\lesssim&h^{s+2}{ \|u\|_{s+2}}\|\Phi\|_2.
\end{split}
\end{equation*}

As to \eqref{L2-estimate-4},   applying the triangle inequality, $s=\min\{j,\ell\}$, the definition of $Q_b$, $\Phi=0$ on $\pa\O$, the Cauchy-Schwarz inequality, the trace inequality  \eqref{trace-inequality}, \eqref{e1}-\eqref{e2}  yields
\begin{equation*}\label{error-equation-6}
\begin{split}
&|\sum_{T\in {\cal T}_h}\langle(I-\mathbb{Q}_{s})({\color{black}{a}}\nabla u)\cdot\textbf{n},Q_0\Phi-Q_b\Phi\rangle_{\pa T}|\\
=&|\sum_{T\in {\cal T}_h}\langle(I-\mathbb{Q}_{s})({\color{black}{a}}\nabla u)\cdot\textbf{n},Q_0\Phi-\Phi\rangle_{\pa T}+\sum_{T\in {\cal T}_h}\langle(I-\mathbb{Q}_{s})({\color{black}{a}}\nabla u)\cdot\textbf{n},\Phi-Q_b\Phi\rangle_{\pa T}|\\
=&|\sum_{T\in {\cal T}_h}\langle(I-\mathbb{Q}_{s})({\color{black}{a}}\nabla u)\cdot\textbf{n},Q_0\Phi-\Phi\rangle_{\pa T}|+|\sum_{T\in {\cal T}_h}\langle {\color{black}{a}}\nabla u\cdot\textbf{n},\Phi-Q_b\Phi\rangle_{\pa T}|\\
=&|\sum_{T\in {\cal T}_h}\langle(I-\mathbb{Q}_{s})({\color{black}{a}}\nabla u)\cdot\textbf{n},Q_0\Phi-\Phi\rangle_{\pa T}|+| \langle {\color{black}{a}}\nabla u\cdot\textbf{n},\Phi-Q_b\Phi\rangle_{\pa \Omega}|\\
\lesssim&\Big(\sum_{T\in {\cal T}_h}\|(I-\mathbb{Q}_{s}){\color{black}{a}}\nabla u\|_{\pa T}^2\Big)^{\frac{1}{2}}\Big(\sum_{T\in {\cal T}_h}\|Q_0\Phi-\Phi\|_{\pa T}^2\Big)^{\frac{1}{2}}\\
\lesssim&\Big(\sum_{T\in {\cal T}_h}h_T^{-1}\|(I-\mathbb{Q}_{s}){\color{black}{a}}\nabla u\|_{T}^2+h_T\|\nabla((I-\mathbb{Q}_{s}){\color{black}{a}}\nabla u)\|_{T}^2\Big)^{\frac{1}{2}}\\
\cdot&\Big(\sum_{T\in {\cal T}_h}h_T^{-1}\|Q_0\Phi-\Phi\|_{T}^2+h_T\|\nabla(Q_0\Phi-\Phi)\|_{T}^2\Big)^{\frac{1}{2}}\\
\lesssim&\begin{cases} h^{s+1}\|u\|_{s+2}\|\Phi\|_1,& if~ k=0,\\
 h^{\frac{-1}{2}}h^{s+1}\|u\|_{s+2}\cdot h^{\frac{-1}{2}}h^{2}\|\Phi\|_2,&  if~ k>0,
         \end{cases}\\
\lesssim&\begin{cases} h^{s+1}{ \|u\|_{s+2}}\|\Phi\|_2,& if~ k=0,\\
h^{s+2}{ \|u\|_{s+2}}\|\Phi\|_2,&  if~ k>0.
         \end{cases}
\end{split}
\end{equation*}

As to \eqref{L2-estimate-5}, for the case of $s\geq\max\{k-1,\ell\}$, it follows from the definition of $\mathbb{Q}_s$ that
\begin{equation*}\label{error-equation-72}
\begin{split}
|\sum_{T\in{\cal T}_h}({\color{black}{a}}\nabla_gQ_hu,(I-\mathbb{Q}_s)\nabla_gQ_h\Phi)_T|=0.
\end{split}
\end{equation*}
For the case of $s<\max\{k-1,\ell\}$, using \eqref{weak gradient-1}, the definition of $\mathbb{Q}_s$, the Cauchy-Schwarz inequality, Lemma \ref{gradient property},   \eqref{e1}-\eqref{e2} and the inverse inequality  gives
{ \begin{equation*}\label{error-equation-7}
\begin{split}
 &|\sum_{T\in{\cal T}_h}({\color{black}{a}}\nabla_gQ_hu,(I-\mathbb{Q}_s)\nabla_gQ_h\Phi)_T|\\
=&|\sum_{T\in{\cal T}_h}(\nabla Q_0u+\delta_g Q_hu,{\color{black}{a}}(I-\mathbb{Q}_s)\nabla_gQ_h\Phi)_T|\\
=&|\sum_{T\in{\cal T}_h}((\nabla Q_0u-\nabla u)+(I-\mathbb{Q}_s)\nabla u+\delta_gQ_hu,{\color{black}{a}}(I-\mathbb{Q}_s)\nabla_gQ_h\Phi)_T|\\
\lesssim&\Big(\sum_{T\in{\cal T}_h}\|\nabla Q_0u-\nabla u\|_T^2+\|(I-\mathbb{Q}_s)\nabla u\|_T^2+\|\delta_gQ_hu\|_T^2\Big)^{\frac{1}{2}}\\
  &\cdot\Big(\sum_{T\in{\cal T}_h}\|{\color{black}{a}}(I-\mathbb{Q}_s)\nabla_gQ_h\Phi\|_{T}^2\Big)^{\frac{1}{2}}\\
\lesssim&\Big(\sum_{T\in{\cal T}_h}\|\nabla Q_0u-\nabla u\|_T^2+\|(I-\mathbb{Q}_s)\nabla u\|_T^2+h_T^{2k}\|u\|^2_{k+1,T}\Big)^{\frac{1}{2}}\\
  &\cdot\Big(\sum_{T\in{\cal T}_h}h_T^2|\nabla_gQ_h\Phi|_{1,T}^2\Big)^{\frac{1}{2}}\\
\lesssim&\begin{cases} \Big(\sum_{T\in{\cal T}_h}\|u\|^2_{1,T}\Big)^{\frac{1}{2}}
\Big(\sum_{T\in{\cal T}_h}\|\delta_gQ_h\Phi\|_{T}^2\Big)^{\frac{1}{2}},~ if~ k=0,~s<\max\{k-1,\ell\},&\\
\Big(\sum_{T\in{\cal T}_h}h_T^{2k}\|u\|^2_{k+1,T}+h_T^{2s+2}\|u\|^2_{s+2,T}+h_T^{2k}\|u\|^2_{k+1,T}\Big)^{\frac{1}{2}}
\Big(\sum_{T\in{\cal T}_h}h_T^2|\nabla Q_0\Phi+\delta_gQ_h\Phi|_{1,T}^2\Big)^{\frac{1}{2}},\\
~~~~~~~~  if~ k>0,~s<\max\{k-1,\ell\},&
\end{cases}\\
\lesssim&\begin{cases}\Big(\sum_{T\in{\cal T}_h}\|\Phi\|_{1,T}^2\Big)^{\frac{1}{2}}\|u\|_1,~ if~ k=0,~s<\max\{k-1,\ell\},&\\
(h^{k}\|u\|_{k+1}+h^{s+1}\|u\|_{s+2})
\Big(\sum_{T\in{\cal T}_h}h_T^2\|\Phi\|_{2,T}^2+h_T^2h_T^{-2}\|\delta_gQ_h\Phi\|_{T}^2\Big)^{\frac{1}{2}},\\
~~~~~~~~  if~ k>0,~s<\max\{k-1,\ell\},&
\end{cases}\\
\lesssim&\begin{cases}\|u\|_1\|\Phi\|_2,~ if~ k=0,~s<\max\{k-1,\ell\},&\\
(h^{k}\|u\|_{k+1}+h^{s+1}\|u\|_{s+2})
\Big(\sum_{T\in{\cal T}_h}h_T^2\|\Phi\|_{2,T}^2+h_T^2\|\Phi\|_{2,T}^2\Big)^{\frac{1}{2}},\\
~~\qquad\qquad\qquad if~ k>0,~s<\max\{k-1,\ell\},&
\end{cases}\\
\lesssim&\begin{cases}\|u\|_1\|\Phi\|_2,~ if~ k=0,~s<\max\{k-1,\ell\},&\\
(h^{k+1}\|u\|_{k+1}+h^{s+2}\|u\|_{s+2})\|\Phi\|_2,
~~~~~~~~  if~ k>0,~s<\max\{k-1,\ell\}.&
\end{cases}
\end{split}
\end{equation*}}
This completes the proof of the lemma.
\end{proof}

\section{Error Estimates}\label{Section:EEs}
The goal of this section is to establish some error estimates for the numerical approximation arising from the {\rm g}WG scheme \eqref{WG-scheme}.

\begin{theorem}\label{THM:energy-estimate}
Let $s=\min\{j,\ell\}$. Assume that the exact solution $u$ of the model problem \eqref{model-problem} is sufficiently regular such that $u\in H^{k+1}(\O)\cap H^{s+2}(\O)$. Let $u_h\in V_h$ be the numerical approximation arising from the {\rm gWG} scheme \eqref{WG-scheme}.  The  error estimate  holds true
\begin{equation*}\label{EQ:wb-pre:new}
\3bare_h\3bar\lesssim
\left\{
\begin{array}{lr}
(h^{\frac{1+\gamma}{2}}+1)\|u\|_1+h^{\frac{2s+1-\gamma}{2}}\|u\|_{s+2} ,\quad if ~k=0,~s=j ~ or ~k=0,~s=\ell,~ \ell>1,\\
h^{\frac{1+\gamma}{2}}\|u\|_1+h^{\frac{2s+1-\gamma}{2}}\|u\|_{s+2} ,\quad if ~ k=0,~s=\ell,~~\ell \leq 1,\\
(h^{\frac{2k+1+\gamma}{2}}+h^k)\|u\|_{k+1}+(h^{\frac{2s+1-\gamma}{2}}+h^{s+1})\|u\|_{s+2}  ,\\
~~~~~~~~if~k>0,s>1,k<s-1,~or~k>0,s<\max\{k-1,\ell\},\\
h^{\frac{2k+1+\gamma}{2}}\|u\|_{k+1}+(h^{\frac{2s+1-\gamma}{2}}+h^{s+1})\|u\|_{s+2}  ,\;~\qquad\qquad\qquad otherwise.
\end{array}
\right.
\end{equation*}
\end{theorem}

\begin{proof}
By taking $v=e_h\in V_h^0$ in the error equation \eqref{Error-equation-22}, one arrives at
 \begin{equation}\label{energy-estimate-2}
 \3bare_h\3bar^2=\zeta_u(e_h).
 \end{equation}
Substituting Lemma \ref{energy-estimate-pre} through setting $\varphi=u$ and $v=e_h$ into the right hand of \eqref{energy-estimate-2} yields
\begin{equation*}
\begin{split}
\3bare_h\3bar^2\lesssim&
\left\{
\begin{array}{lr}
(h^{\frac{1+\gamma}{2}}+1)\|u\|_1\3bare_h\3bar+h^{\frac{2s+1-\gamma}{2}}\|u\|_{s+2}\3bare_h\3bar,
~if~k=0,~s=j,~j\leq1,\\
(h^{\frac{1+\gamma}{2}}+1)\|u\|_1\3bare_h\3bar+h^{\frac{2s+1-\gamma}{2}}\|u\|_{s+2}\3bare_h\3bar,
~if~k=0,~s=j,~j>1,\\
h^{\frac{1+\gamma}{2}}\|u\|_1\3bare_h\3bar+h^{\frac{2s+1-\gamma}{2}}\|u\|_{s+2}\3bare_h\3bar,
~if~k=0,~s=\ell,~\ell\leq1,\\
(h^{\frac{1+\gamma}{2}}+1)\|u\|_1\3bare_h\3bar+h^{\frac{2s+1-\gamma}{2}}\|u\|_{s+2}\3bare_h\3bar,
~if~k=0,~s=\ell,~\ell>1,\\
h^{\frac{2k+1+\gamma}{2}}\|u\|_{k+1}\3bare_h\3bar+h^k\|u\|_{k+1}\3bare_h\3bar
 +(h^{\frac{2s+1-\gamma}{2}}+h^{s+1})\|u\|_{s+2}\3bare_h\3bar,\\
 ~~~~~~~~~~~if~k>0,s>1,k<s-1,~or~k>0,s<\max\{k-1,\ell\},\\
h^{\frac{2k+1+\gamma}{2}}\|u\|_{k+1}\3bare_h\3bar
 +(h^{\frac{2s+1-\gamma}{2}}+h^{s+1})\|u\|_{s+2}\3bare_h\3bar,~otherwise,\\
\end{array}
\right.\\
\lesssim&
\left\{
\begin{array}{lr}
((h^{\frac{1+\gamma}{2}}+1)\|u\|_1+h^{\frac{2s+1-\gamma}{2}}\|u\|_{s+2})\3bare_h\3bar,\quad if ~k=0,~s=j~ or ~k=0,~s=\ell, ~\ell>1,\\
(h^{\frac{1+\gamma}{2}}\|u\|_1+h^{\frac{2s+1-\gamma}{2}}\|u\|_{s+2})\3bare_h\3bar,\quad if  ~k=0,~s=\ell,~ \ell \leq 1,\\
((h^{\frac{2k+1+\gamma}{2}}+h^k)\|u\|_{k+1}+(h^{\frac{2s+1-\gamma}{2}}+h^{s+1})\|u\|_{s+2})\3bare_h\3bar,\\
~~~~~~~~if~k>0,s>1,k<s-1,~or~k>0,s<\max\{k-1,\ell\},\\
(h^{\frac{2k+1+\gamma}{2}}\|u\|_{k+1}+(h^{\frac{2s+1-\gamma}{2}}+h^{s+1})\|u\|_{s+2})\3bare_h\3bar,\;~\qquad\qquad\qquad otherwise,
\end{array}
\right.
\end{split}
\end{equation*}
which leads to the desired error estimate. This completes the proof of the Theorem.
\end{proof}

  { \begin{remark}
    Theorem \ref{THM:energy-estimate} implies that our {\em g}WG scheme \ref{WG-scheme} achieves a superconvergence order of $\mathcal{O}(h)$ in a discrete norm for the case of $k=0$, $s=\ell$, $\ell=1$, $\gamma=1$ and  an optimal convergence order $\mathcal{O}(h^k)$ for the case of  $k >0$, $s=k-1$, $\gamma=-1$.
    \end{remark}}

We shall derive an error estimate for the numerical approximation in the usual $L^2$ norm by using the standard duality argument. To this end, we shall consider the following dual problem that seeks $\Phi \in H^2(\Omega)$ satisfying 
 \begin{equation}\label{dual-equation}
 \begin{split}
-\nabla\cdot({\color{black}{a}}\nabla \Phi)&=e_0,~~ \text{in}~\O,\\
\Phi&=0,~~~~\pa\O.
\end{split}
\end{equation}
We assume that the dual problem \eqref{dual-equation} satisfies the $H^2$ regularity property in the sense that there exists a positive constant $C$ such that
 \begin{equation}\label{dual-regular}
\|\Phi\|_2\leq C\|e_0\|.
\end{equation}

\begin{theorem}\label{THM:L20-estimate}
Let $s=\min\{j,\ell\}$. Let $u_h\in V_h$ and $u\in H^{k+1}(\O)\cap H^{s+2}(\O)$ be the numerical solution of the {\rm gWG} scheme \eqref{WG-scheme} and the exact solution of the model problem \eqref{model-problem}, respectively. In addition, assume that the dual problem \eqref{dual-equation} satisfies the $H^2$ regularity property \eqref{dual-regular}. Then, the following error estimate  holds true
\begin{equation*}
\|e_0\|\lesssim
\left\{
\begin{array}{lr}
(h^{\frac{1+\gamma}{2}}+h^{\frac{1-\gamma}{2}})\3bare_h\3bar,\quad if~k=0,~~s=\ell,~~\ell\leq1,\\
\3bare_h\3bar,\;\quad if~k=0,~~s=\ell,~~\ell>1,~or~k=0,~s=j,\\
(h^{\frac{3+\gamma}{2}}+h^{\frac{1-\gamma}{2}}+h)\3bare_h\3bar,~if~k>0,\\
\end{array}
\right.
\end{equation*}
\end{theorem}
where $\3bare_h\3bar$ is given by Theorem \ref{THM:energy-estimate}.

\begin{proof}
Testing the dual equation \eqref{dual-equation} against $e_0$ and using the usual integration by parts, we have
\begin{equation}\label{error-equation-01}
\begin{split}
\|e_0\|^2
=&(-\nabla\cdot({\color{black}{a}}\nabla\Phi),e_0)\\
=&\sum_{T\in{\cal T}_h}({\color{black}{a}}\nabla\Phi,\nabla e_0)_T
  -\langle {\color{black}{a}}\nabla \Phi\cdot\textbf{n},e_0\rangle_{\pa T}\\
=&\sum_{T\in{\cal T}_h}((I-\mathbb{Q}_s)({\color{black}{a}}\nabla\Phi),\nabla e_0)_T
  -\langle {\color{black}{a}}\nabla \Phi\cdot\textbf{n},e_0-e_b\rangle_{\pa T}\\&+(\mathbb{Q}_s({\color{black}{a}}\nabla\Phi),\nabla e_0)_T,
\end{split}
\end{equation}
where we  used the fact $\sum_{T\in{\cal T}_h}\langle{\color{black}{a}}\nabla\Phi\cdot\textbf{n},e_b\rangle_{\pa T}=0$ since $e_b=0$ on $\pa\O$.

{ 
To deal with the last term on the last line in \eqref{error-equation-01}, choosing $u=\Phi$ and $v=e_h$, from the third line in \eqref{error-equation-1-1},   \eqref{error-equation-1},  and $s=\min\{j, \ell\}$,  we obtain
\begin{equation}\label{error-equation-02}
\begin{split}
&\sum_{T\in{\cal T}_h}(\mathbb{Q}_s({\color{black}{a}}\nabla\Phi),\nabla e_0)_T\\
=&\sum_{T\in{\cal T}_h}(\nabla\Phi,{\color{black}{a}}\mathbb{Q}_s\nabla_ge_h)_T
  -\langle e_b-Q_be_0,\mathbb{Q}_s({\color{black}{a}}\nabla\Phi)\cdot\textbf{n}\rangle_{\pa T}\\
=&\sum_{T\in{\cal T}_h}(\nabla_gQ_h\Phi,{\color{black}{a}}\mathbb{Q}_s\nabla_ge_h)_T
 -(\Phi-Q_0\Phi,\nabla\cdot({\color{black}{a}}\mathbb{Q}_s\nabla_ge_h))_T\\
  &-\langle e_b-e_0,\mathbb{Q}_s({\color{black}{a}}\nabla\Phi)\cdot\textbf{n}\rangle_{\pa T}\\
=&\sum_{T\in{\cal T}_h}({\color{black}{a}}\nabla_gQ_h\Phi,(\mathbb{Q}_s-I)\nabla_ge_h)_T+({\color{black}{a}}\nabla_gQ_h\Phi,\nabla_ge_h)_T\\
 &-(\Phi-Q_0\Phi,\nabla\cdot({\color{black}{a}}\mathbb{Q}_s\nabla_ge_h))_T
  -\langle e_b-e_0,\mathbb{Q}_s({\color{black}{a}}\nabla\Phi)\cdot\textbf{n}\rangle_{\pa T}.
\end{split}
\end{equation}
}
Substituting \eqref{error-equation-02} into \eqref{error-equation-01} and using the error equation \eqref{Error-equation-22} with $v=Q_h\Phi\in V_h^0$, we get
\begin{equation}\label{error-equation-2}
\begin{split}
\|e_0\|^2
=&\sum_{T\in{\cal T}_h}((I-\mathbb{Q}_s)({\color{black}{a}}\nabla\Phi),\nabla e_0)_T
  +\langle(\mathbb{Q}_s-I)({\color{black}{a}}\nabla\Phi)\cdot\textbf{n},e_0-e_b\rangle_{\pa T}\\
 &+({\color{black}{a}}\nabla_gQ_h\Phi,(\mathbb{Q}_s-I)\nabla_ge_h)_T-(\Phi-Q_0\Phi,\nabla\cdot({\color{black}{a}}\mathbb{Q}_s\nabla_ge_h))_T\\
  &+\zeta_u(Q_h\Phi)-s(e_h,Q_h\Phi)\\
=&-\zeta_\Phi(e_h)+\zeta_u(Q_h\Phi)+s(Q_h\Phi,e_h)-s(e_h,Q_h\Phi)\\
=&-\zeta_\Phi(e_h)+\zeta_u(Q_h\Phi),
\end{split}
\end{equation}
where the terms $\zeta_\Phi(e_h)$ and $\zeta_u(Q_h\Phi)$ are given by \eqref{error equation-remainder}.

{ Next, it suffices to deal with the two terms on the last line in \eqref{error-equation-2}. As to the first term $\zeta_\Phi(e_h)$, using \eqref{error equation-remainder}, Lemma \ref{energy-estimate-pre} with $\varphi=\Phi$, $v=e_h$ and \eqref{dual-regular}, there holds
\begin{equation}\label{error-equation-55}
\begin{split}
&|\zeta_\Phi(e_h)|\\
\lesssim&
\left\{
\begin{array}{lr}
(h^{\frac{1+\gamma}{2}}+1)\|\Phi\|_1\3bare_h\3bar+h^{\frac{1-\gamma}{2}}\|\Phi\|_{2}\3bare_h\3bar,
~if~k=0,~s=j,~j\leq1,\\
(h^{\frac{1+\gamma}{2}}+1)\|\Phi\|_1\3bare_h\3bar+h^{\frac{1-\gamma}{2}}\|\Phi\|_{2}\3bare_h\3bar,
~if~k=0,~~s=j,~j>1,\\
h^{\frac{1+\gamma}{2}}\|\Phi\|_1\3bare_h\3bar+h^{\frac{1-\gamma}{2}}\|\Phi\|_{2}\3bare_h\3bar,
~if~k=0,~~s=\ell,~\ell\leq1,\\
(h^{\frac{1+\gamma}{2}}+1)\|\Phi\|_1\3bare_h\3bar+h^{\frac{1-\gamma}{2}}\|\Phi\|_{2}\3bare_h\3bar,
~if~k=0,~~s=\ell,~\ell>1,\\
(h^{\frac{3+\gamma}{2}}+h+h^{\frac{1-\gamma}{2}})\|\Phi\|_2\3bare_h\3bar,~if~k>0,
\end{array}
\right.\\
\lesssim&
\left\{
\begin{array}{lr}
h^{\frac{1+\gamma}{2}}\|\Phi\|_1\3bare_h\3bar+h^{\frac{1-\gamma}{2}}\|\Phi\|_{2}\3bare_h\3bar,\quad if~k=0,~s=\ell,~~\ell\leq1,\\
(h^{\frac{1+\gamma}{2}}+1)\|\Phi\|_1\3bare_h\3bar+h^{\frac{1-\gamma}{2}}\|\Phi\|_{2}\3bare_h\3bar,\;~\quad if~k=0,~~s=\ell,~~\ell>1,~or~k=0,~s=j,\\
(h^{\frac{3+\gamma}{2}}+h^{\frac{1-\gamma}{2}}+h)\|\Phi\|_{2}\3bare_h\3bar,~if~k>0,
\end{array}
\right.\\
\lesssim&
\left\{
\begin{array}{lr}
(h^{\frac{1+\gamma}{2}}+h^{\frac{1-\gamma}{2}})\|e_0\|\3bare_h\3bar,\quad if~k=0,~~s=\ell,~~\ell\leq1,\\
(h^{\frac{1+\gamma}{2}}+h^{\frac{1-\gamma}{2}}+1) \|e_0\|\3bare_h\3bar,\;\quad if~k=0,~~s=\ell,~~\ell>1,~or~k=0,~s=j,\\
(h^{\frac{3+\gamma}{2}}+h^{\frac{1-\gamma}{2}}+h)\|e_0\|\3bare_h\3bar,~if~k>0.
\end{array}
\right.
\end{split}
\end{equation}
As to the second term $\zeta_u(Q_h\Phi)$, from Lemma \ref{L2-estimate-pre} and \eqref{dual-regular}, we have
\begin{equation}\label{error-equation-44}
\begin{split}
&|\zeta_u(Q_h\Phi)|\\
\lesssim&
\left\{
\begin{array}{lr}
((h^{\gamma+1}+1)\|u\|_1+h^{s+1}\|u\|_{s+2})\|\Phi\|_2,~if~k=0,~s=j,~j\leq1,\\
((h^{\gamma+1}+1)\|u\|_1+h^{s+1}\|u\|_{s+2})\|\Phi\|_2,~if~k=0,~s=j,~j>1,\\
(h^{\gamma+1}\|u\|_1+h^{s+1}\|u\|_{s+2})\|\Phi\|_2,~if~k=0,~s=\ell,~\ell\leq1,\\
((h^{\gamma+1}+1)\|u\|_1+h^{s+1}\|u\|_{s+2})\|\Phi\|_2,~if~k=0,~s=\ell,~\ell>1,\\
((h^{k+\gamma+2}+h^{k+1})\|u\|_{k+1}+h^{s+2}\|u\|_{s+2})\|\Phi\|_2,\\~~~~~~~~if~k>0,s>1,k<s-1,~or~k>0,s<\max\{k-1,\ell\},\\
(h^{k+\gamma+2}\|u\|_{k+1}+h^{s+2}\|u\|_{s+2})\|\Phi\|_2,~~~~~~~~otherwise,
\end{array}
\right.\\
\lesssim&
\left\{
\begin{array}{lr}
(h^{\gamma+1}\|u\|_1+h^{s+1}\|u\|_{s+2})\|e_0\|,\quad if~k=0,~~s=\ell,~~\ell\leq1,\\
((h^{\gamma+1}+1)\|u\|_1+h^{s+1}\|u\|_{s+2})\|\|e_0\|,\quad k=0,~~s=\ell,~~\ell>1,~or~k=0,~s=j,\\
(h^{k+1}(1+h^{\gamma+1})\|u\|_{k+1}+h^{s+2}\|u\|_{s+2})\|e_0\|,\\
~~~~~~~~if~k>0,s>1,k<s-1,~or~k>0,s<\max\{k-1,\ell\},\\
(h^{k+\gamma+2}\|u\|_{k+1}+h^{s+2}\|u\|_{s+2})\|e_0\|,\;~\qquad otherwise.
\end{array}
\right.
\end{split}
\end{equation}}
Substituting      \eqref{error-equation-55}-\eqref{error-equation-44} into \eqref{error-equation-2} leads to the desired results. This completes the proof.
\end{proof}

  { \begin{remark}
     Theorem \ref{THM:L20-estimate} implies that   $\|e_0\|$ has 
    an optimal  convergence order of $\mathcal{O}(h^{k+1})$ for the case of $k=0$, $s=\ell$, $\ell=1$, $\gamma=1$ and the case of $k\geq1$, $s=k-1$, $\gamma=-1$.
    \end{remark}}
  
To establish the error estimates for $e_b$, we introduce the following norm
$$
\|e_b\|_{\E_h}=\Big(\sum_{T\in{\cal T}_h}h_T\|e_b\|_{\pa T}^2\Big)^{\frac{1}{2}}.
$$

\begin{theorem}\label{THM:L2b-estimate}
In the assumptions of Theorem \ref{THM:energy-estimate}, we have the following error estimate
\begin{equation*}
\|e_b\|_{\E_h}\lesssim
\left\{
\begin{array}{l}
(h^{\frac{1+\gamma}{2}}+h^{\frac{1-\gamma}{2}})\3bare_h\3bar,~~\qquad \qquad if~k=0,~s=\ell,~\ell\leq1,\\
(1+h^{\frac{1-\gamma}{2}})\3bare_h\3bar,\qquad\qquad\qquad~ if~k=0,~s=\ell, ~\ell>1,~or~k=0,s=j,\\
(h^{\frac{3+\gamma}{2}}+h^{\frac{1-\gamma}{2}}+h)\3bare_h\3bar,\qquad\qquad\quad if~k>0,
\end{array}
\right.
\end{equation*}
\end{theorem}
where $\3bare_h\3bar$ is given by Theorem \ref{THM:energy-estimate}.

\begin{proof}
It follows from the triangle inequality, the trace inequality \eqref{inverse-inequality} and \eqref{three-bar} that
\begin{equation*}
\begin{split}
\sum_{T\in{\cal T}_h}h_T\|e_b\|_{\pa T}^2
\lesssim&\sum_{T\in{\cal T}_h}h_T\|Q_be_0\|_{\pa T}^2+h_T\|e_b-Q_be_0\|_{\pa T}^2\\
\lesssim&\|Q_be_0\|^2+h^{1-\gamma}\3bare_h\3bar^2\\
\lesssim&\|e_0\|^2+h^{1-\gamma}\3bare_h\3bar^2,
\end{split}
\end{equation*}
which, together with Theorems \ref{THM:energy-estimate}-\ref{THM:L20-estimate}, leads to Theorem \ref{THM:L2b-estimate}. This completes the proof of the theorem.
\end{proof}

\section{Numerical Experiments}\label{Section:EE}

In this section, a series of numerical experiments are presented to verify  the convergence theory established in the previous sections.

Recall that the generalized discrete weak gradient is computed by \eqref{weak gradient-1}-\eqref{weak gradient-2}. For the simplicity of notation, the {\rm g}WG  element and the generalized discrete weak gradient are denoted by $P_k(T)/P_j(\pa T)/[P_\ell(T)]^2$ element. We choose $a$ in the model problem \eqref{model-problem} to be an identity matrix on the unit square domain $\O=(0,1)^2$. The uniform triangular partition and the uniform rectangular partition are employed. The uniform triangular partition is obtained through a successive refinement of an initial triangular partition of the domain $\O$ by   connecting the middle points of the edges of each triangular element. The uniform rectangular partition is generated from an initial $3\times 2$ rectangular partition of $\O$ with the next level of partition being obtained by connecting the middle points on the two parallel edges. 
In all tables,   ``Theory.rate'' means the the convergence theory established in this paper; and ``N/A'' means the convergence rate that has not been developed in this paper.




\subsection{The {\rm gWG} elements on the uniform triangular partition with smooth solutions}

In this section, the uniform triangular partition is employed and the exact solution is chosen as $u=\cos(\pi x)\cos(\pi y)$.


Table \ref{NE:TRI:v0} illustrates the performance of the $P_k(T)/P_{4}(\pa T)/[P_{4}(T)]^2$ elements with $k=3,4,5,6$. The stabilization parameters are given by $\rho=1$ and $\gamma=-1$. 
For $k=3,5,6$, the numerical  convergence rate   is in a good consistency with the theoretical convergence rate for $\3bare_h\3bar$, $\|e_0\|$ and $\|e_b\|_{\E_h}$ respectively. For $k=4$, the numerical  convergence rate is  consistent with the theoretical convergence rate for $\|e_b\|_{\E_h}$, and is higher than the theoretical convergence rate for $\3bare_h\3bar$ and $\|e_0\|$. 

\begin{table}[htbp]\centering\scriptsize
\begin{center}
\caption{Numerical errors and convergence rates for the $P_k(T)/P_{4}(\pa T)/[P_{4}(T)]^2$ elements.}\label{NE:TRI:v0}
\begin{tabular}{p{1.5cm}p{1.9cm}p{1.0cm}p{1.9cm}p{1.0cm}p{2cm}p{0.5cm}}
\hline
                         $1/h$     &$\3bare_h\3bar$  &Rate       &$\|e_0\|$ &Rate       &$\|e_b\|_{\E_h}$ &Rate\\ \hline
                                   & $k=3$           &           &          &           &                 &\\ \hline
                           8       &1.56E-04         &3.00       &1.33E-06  &4.11       &4.13E-06         &3.99\\
                          16       &1.95E-05         &3.00       &8.03E-08  &4.04       &2.60E-07         &3.99\\
                          32       &2.45E-06         &3.00       &4.96E-09  &4.02       &1.63E-08         &3.99\\
                          64       &3.06E-07         &3.00       &3.08E-010 &4.01       &1.02E-09         &4.00\\
\text{Theory.rate}                 &                 &3.0        &          &4.0        &                 &4.0\\
\hline
                                   & $k=4$           &           &          &           &                 &\\ \hline
                           4       &1.85E-04         &4.86       &9.12E-06  &5.90       &1.55E-06         &5.61\\
                           8       &6.58E-06         &4.81       &1.49E-07  &5.94       &3.79E-08         &5.35\\
                          16       &2.79E-07         &4.56       &2.59E-09  &5.85       &1.09E-09         &5.12\\
                          32       &1.47E-08         &4.25       &5.40E-011 &5.58       &3.34E-011        &5.03\\
\text{Theory.rate}                 &                 &4.0        &          &5.0        &                 &5.0\\ \hline
                                   &$k=5$            &           &          &           &                 &\\ \hline
                           2       &5.89E-03         &5.07       &6.89E-04  &6.08       &5.15E-05         &5.88\\
                           4       &1.94E-04         &4.92       &1.13E-05  &5.92       &9.22E-07         &5.80\\
                           8       &6.15E-06         &4.98       &1.80E-07  &5.98       &1.50E-08         &5.94\\
                          16       &1.93E-07         &4.99       &2.81E-09  &5.99       &2.39E-010        &5.97\\
\text{Theory.rate}                 &                 &5.0        &          &6.0        &                 &6.0\\  \hline

                                   & $k=6$           &           &          &           &                 &\\ \hline
                           2       &6.07E-03         &5.07       &7.20E-04  &6.08       &5.32E-05         &5.88\\
                           4       &2.00E-04         &4.92       &1.19E-05  &5.92       &9.55E-07         &5.80\\
                           8       &6.34E-06         &4.98       &1.88E-07  &5.98       &1.56E-08         &5.94\\
                          16       &1.99E-07         &4.99       &2.95E-09  &5.99       &2.48E-010        &5.98\\
\text{Theory.rate}                 &                 &5.0        &          &6.0        &                 &6.0\\
\hline
\end{tabular}
\end{center}
\end{table}
 
  The numerical results for the $P_0(T)/P_j(\pa T)/[P_\ell(T)]^2$ element with $\gamma=0$ and $\gamma=1$ are shown in Tables \ref{NE:Tri-2}-\ref{NE:Tri-3}. The stabilization parameter is $\rho=1$. We can see from Table \ref{NE:Tri-2} that the numerical results consist with the theoretical rates of convergence. In addition, for the $P_0(T)/P_j(\pa T)/[P_1(T)]^2$ for $j=1,2$ elements, the convergence rate for $\|e_0\|$ seems to be in an convergence order of ${\cal O}(h)$ which exceeds the theoretical prediction  ${\cal O}(h^{0.5})$. We observe from Table \ref{NE:Tri-3} that the numerical performance of {\rm g}WG methods is typically better than what the theory predicts.

\begin{table}[htbp]\centering\scriptsize
\begin{center}
\caption{Numerical errors and convergence rates for the $P_0(T)/P_j(\pa T)/[P_\ell(T)]^2$ elements when $\gamma=0$.} \label{NE:Tri-2}
\begin{tabular}{p{1.4cm}p{1.8cm}p{0.8cm}p{2cm}p{0.8cm}p{2cm}p{0.5cm}}
\hline
                         $1/h$     &$\3bare_h\3bar$  &Rate      &$\|e_0\|$ &Rate      &$\|e_b\|_{\E_h}$ &Rate\\
\hline
                                   &$j=0$,~$\ell=0$  &          &          &          &                 &           \\ \hline
                           16      &9.35E-01         &0.48      &8.77E-02  &0.99      &9.17E-03         &1.09\\
                           32      &6.65E-01         &0.49      &4.42E-02  &1.00      &4.46E-03         &1.04\\
                           64      &4.71E-01         &0.50      &2.21E-02  &1.00      &2.20E-03         &1.02\\
                          128      &3.34E-01         &0.50      &1.11E-02  &1.00      &1.09E-03         &1.01\\
\text{Theory.rate}                 &                 &0.5       &          &1.0       &                 &1.0\\ \hline
                                   &$j=1$,~$\ell=0$  &          &          &          &                 &           \\ \hline
                           16      &9.96E-01         &0.48      &8.77E-02  &0.96      &5.55E-02         &0.94\\
                           32      &7.09E-01         &0.49      &4.42E-02  &0.99      &2.82E-02         &0.98\\
                           64      &5.03E-01         &0.50      &2.21E-02  &1.00      &1.42E-02         &0.99\\
                          128      &3.56E-01         &0.50      &1.11E-02  &1.00      &7.14E-03         &0.99\\
\text{Theory.rate}                 &                 &0.5       &          &1.0       &                 &1.0\\ \hline
                                   &$j=1$,~$\ell=1$  &          &          &          &                 &         \\ \hline
                           16      &8.86E-02         &1.01      &1.27E-02  &0.98      &2.23E-02         &1.01\\
                           32      &4.39E-02         &1.01      &6.37E-03  &1.00      &1.11E-02         &1.01\\
                           64      &2.18E-02         &1.01      &3.18E-03  &1.00      &5.50E-03         &1.01\\
                          128      &1.09E-02         &1.01      &1.58E-03  &1.00      &2.74E-03         &1.00\\
\text{Theory.rate}                 &                 &0.5       &          &1.0       &                 &1.0\\ \hline
                                   &$j=2$,~$\ell=1$  &          &          &          &                 &           \\ \hline
                           16      &8.98E-02         &1.03      &1.27E-02  &0.98      &2.24E-02         &1.04\\
                           32      &4.42E-02         &1.02      &6.37E-03  &1.00      &1.11E-02         &1.02\\
                           64      &2.19E-02         &1.01      &3.18E-03  &1.00      &5.50E-03         &1.01\\
                          128      &1.09E-02         &1.01      &1.58E-03  &1.00      &2.75E-03         &1.00\\
\text{Theory.rate}                 &                 &0.5       &          &1.0       &                 &1.0\\ \hline
                                   &$j=2$,~$\ell=3$  &          &          &          &                 &           \\ \hline
                           16      &6.20E-00         &-0.03     &1.62E-01  &-0.01     &2.85E-01         &0.05\\
                           32      &6.26E-00         &-0.01     &1.62E-01  &0.00      &2.81E-01         &0.02\\
                           64      &6.28E-00         &-0.01     &1.62E-01  &0.00      &2.80E-01         &0.01\\
                          128      &6.30E-00         &-0.00     &1.61E-01  &0.00      &2.80E-01         &0.00\\
\text{Theory.rate}                 &                 &0.0       &          &0.0       &                 &0.0\\
\hline
\end{tabular}
\end{center}
\end{table}

\begin{table}[htbp]\centering\scriptsize
\begin{center}
\caption{Numerical errors and convergence rates for the $P_0(T)/P_j(\pa T)/[P_\ell(T)]^2$ elements when $\gamma=1$.}\label{NE:Tri-3}
\begin{tabular}{p{1.4cm}p{1.8cm}p{0.8cm}p{2cm}p{0.8cm}p{2cm}p{0.6cm}}
\hline
                         $1/h$     &$\3bare_h\3bar$  &Rate       &$\|e_0\|$ &Rate       &$\|e_b\|_{\E_h}$ &Rate\\
\hline
                                   &$j=0$,~$\ell=0$  &           &          &           &                 &           \\ \hline
                           16      &3.17E-00         &-0.01      &1.02E-00  &-0.01      &1.98E-03         &1.99\\
                           32      &3.17E-00         &-0.00      &1.02E-00  &-0.00      &4.95E-04         &2.00\\
                           64      &3.18E-00         &-0.00      &1.02E-00  &-0.00      &1.24E-04         &2.00\\
                          128      &3.18E-00         &-0.00      &1.02E-00  &-0.00      &3.09E-05         &2.00\\
\text{Theory.rate}                 &                 &0.0        &          &0.0        &                 &0.0\\ \hline
                                   &$j=1$,~$\ell=0$  &           &          &           &                 &\\ \hline
                           16      &3.17E-00         &-0.01      &1.02E-00  &-0.01      &5.48E-02         &0.94\\
                           32      &3.17E-00         &-0.00      &1.02E-00  &-0.00      &2.79E-02         &0.97\\
                           64      &3.18E-00         &-0.00      &1.02E-00  &-0.00      &1.41E-02         &0.99\\
                          128      &3.18E-00         &-0.00      &1.02E-00  &-0.00      &7.06E-03         &0.99\\
\text{Theory.rate}                 &                 &0.0        &          &0.0        &                 &0.0\\ \hline
                                   &$j=1$,~$\ell=1$  &           &          &           &                 &\\ \hline
                           16      &1.46E-02         &1.93       &2.22E-03  &1.92       &3.88E-03         &1.96\\
                           32      &3.73E-03         &1.97       &5.62E-04  &1.98       &9.76E-04         &1.99\\
                           64      &9.38E-04         &1.99       &1.41E-04  &1.99       &2.45E-04         &2.00\\
                          128      &2.35E-04         &2.00       &3.53E-05  &2.00       &6.12E-05         &2.00\\
\text{Theory.rate}                 &                 &1.0        &          &1.0        &                 &1.0\\ \hline
                                   &$j=2$,~$\ell=1$  &           &          &           &                 &\\ \hline
                           16      &1.53E-02         &1.93       &2.22E-03  &1.92       &4.48E-03         &1.96\\
                           32      &3.88E-03         &1.98       &5.62E-04  &1.98       &1.13E-03         &1.99\\
                           64      &9.77E-04         &1.99       &1.41E-04  &1.99       &2.83E-04         &2.00\\
                          128      &2.45E-04         &2.00       &3.53E-05  &2.00       &7.07E-05         &2.00\\
\text{Theory.rate}                 &                 &1.0        &          &1.0        &                 &1.0\\ \hline
                                   &$j=2$,~$\ell=3$  &           &          &           &                 &\\ \hline
                           16      &6.19E-00         &-0.03      &1.62E-01  &-0.01      &2.85E-01         &0.05\\
                           32      &6.25E-00         &-0.01      &1.62E-01  &0.00       &2.81E-01         &0.02\\
                           64      &6.28E-00         &-0.01      &1.62E-01  &0.00       &2.80E-01         &0.01\\
                          128      &6.29E-00         &-0.00      &1.61E-01  &0.00       &2.80E-01         &0.00\\
\text{Theory.rate}                 &                 &0.0        &          &0.0        &                 &0.0\\
\hline
\end{tabular}
\end{center}
\end{table}

Table \ref{NE:Tri-4} reports some numerical results of the $P_5(T)/P_j(\pa T)/[P_1(T)]^2$ elements for different values of $j$. The stabilization parameters are $\rho=1$ and $\gamma=-1$. These numerical results are greatly consistent with the established theory. The convergence rates for $\3bare_h\3bar$, $\|e_0\|$ and $\|e_b\|_{\E_h}$ are the same when different $j=2, 5,6,7$ are applied. 


\begin{table}[htbp]\centering\scriptsize
\begin{center}
\caption{Numerical errors and convergence rates for the $P_5(T)/P_j(\pa T)/[P_{1}(T)]^2$ elements.}
\label{NE:Tri-4}
\begin{tabular}{p{1.4cm}p{1.8cm}p{1.0cm}p{2cm}p{1.0cm}p{2cm}p{0.6cm}}
\hline
                         $1/h$    &$\3bare_h\3bar$  &Rate      &$\|e_0\|$ &Rate      &$\|e_b\|_{\E_h}$ &Rate\\
\hline
                                  &$j=0$            &          &          &          &                 &\\ \hline
                          8       &2.53E-01         &1.02      &6.75E-03  &2.12      &6.61E-03         &1.90\\
                          16      &1.26E-01         &1.01      &1.65E-03  &2.03      &1.68E-03         &1.98\\
                          32      &6.30E-02         &1.00      &4.09E-04  &2.01      &4.21E-04         &1.99\\
                          64      &3.15E-02         &1.00      &1.02E-04  &2.00      &1.05E-04         &2.00\\
\text{Theory.rate}                &                 &1.0       &          &2.0       &                 &2.0\\ \hline
                                  &$j=2$            &          &          &          &                 &\\ \hline
                          8       &5.52E-02         &1.98      &2.23E-03  &3.00      &5.32E-04         &3.67\\
                          16      &1.38E-02         &2.00      &2.78E-04  &3.00      &4.55E-05         &3.55\\
                          32      &3.46E-03         &2.00      &3.48E-05  &3.00      &4.67E-06         &3.28\\
                          64      &8.66E-04         &2.00      &4.34E-06  &3.00      &5.46E-07         &3.10\\
\text{Theory.rate}                &                 &2.0       &          &3.0       &                 &3.0\\ \hline
                                  &$j=5$            &          &          &          &                 &\\ \hline
                          8       &5.52E-02         &1.98      &2.23E-03  &3.00      &5.71E-04         &3.61\\
                          16      &1.38E-02         &2.00      &2.78E-04  &3.00      &5.25E-05         &3.44\\
                          32      &3.46E-03         &2.00      &3.48E-05  &3.00      &5.72E-06         &3.20\\
                          64      &8.66E-04         &2.00      &4.34E-06  &3.00      &6.85E-07         &3.06\\
\text{Theory.rate}                &                 &2.0       &          &3.0       &                 &3.0\\ \hline
                                  &$j=6$            &          &          &          &                 &\\ \hline
                          8       &5.52E-02         &1.98      &2.23E-03  &3.00      &5.71E-04         &3.61\\
                          16      &1.38E-02         &2.00      &2.78E-04  &3.00      &5.25E-05         &3.44\\
                          32      &3.46E-03         &2.00      &3.48E-05  &3.00      &5.72E-06         &3.20\\
                          64      &8.66E-04         &2.00      &4.34E-06  &3.00      &6.85E-07         &3.06\\
\text{Theory.rate}                &                 &2.0       &          &3.0       &                 &3.0\\ \hline
                                  &$j=7$            &          &          &          &                 &\\ \hline
                          8       &5.52E-02         &1.98      &2.23E-03  &3.00      &5.71E-04         &3.61\\
                          16      &1.38E-02         &2.00      &2.78E-04  &3.00      &5.25E-05         &3.44\\
                          32      &3.46E-03         &2.00      &3.48E-05  &3.00      &5.72E-06         &3.20\\
                          64      &8.66E-04         &2.00      &4.34E-06  &3.00      &6.85E-07         &3.06\\
\text{Theory.rate}                &                 &2.0       &          &3.0       &                 &3.0\\
\hline
\end{tabular}
\end{center}
\end{table}

Table \ref{NE:TRI:THIRD2} demonstrates the performance of the $P_4(T)/P_2(\pa T)/[P_\ell(T)]^2$ elements for different values of $\ell$. The stabilization parameters are given by $\rho=1$ and $\gamma=-1$. We observe that these numerical results are in an agreement with our theory. Moreover,  the convergence rates for the numerical approximations are the same for different $\ell=3, 4, 6$.

\begin{table}[htbp]\centering\scriptsize
\begin{center}
\caption{Numerical errors and convergence rates for the $P_4(T)/P_{2}(\pa T)/[P_\ell(T)]^2$ elements.} \label{NE:TRI:THIRD2}
\begin{tabular}{p{1.4cm}p{1.8cm}p{1.0cm}p{2cm}p{1.0cm}p{2cm}p{0.6cm}}
\hline
                         $1/h$     &$\3bare_h\3bar$  &Rate      &$\|e_0\|$ &Rate      &$\|e_b\|_{\E_h}$  &Rate\\
\hline
                                   &$\ell=1$         &          &          &          &                  &\\ \hline
                           8       &5.52E-02         &1.98      &2.23E-03  &3.00      &5.33E-04          &3.67\\
                          16       &1.38E-02         &2.00      &2.78E-04  &3.00      &4.57E-05          &3.55\\
                          32       &3.46E-03         &2.00      &3.47E-05  &3.00      &4.67E-06          &3.28\\
                          64       &8.66E-03         &2.00      &4.34E-06  &3.00      &5.45E-07          &3.10\\
\text{Theory.rate}                 &                 &2.0       &          &3.0       &                  &3.0\\ \hline
                                   &$\ell=3$         &          &          &          &                  &\\ \hline
                           8       &2.82E-04         &3.32      &5.42E-06  &4.76      &1.57E-06          &3.94\\
                          16       &3.25E-05         &3.11      &2.43E-07  &4.48      &9.87E-08          &3.99\\
                          32       &3.98E-06         &3.03      &1.32E-08  &4.20      &6.18E-09          &4.00\\
                          64       &4.95E-07         &3.01      &7.93E-010 &4.06      &3.86E-010         &4.00\\
\text{Theory.rate}                 &                 &3.0       &          &4.0       &                  &4.0\\ \hline
                                   &$\ell=4$         &          &          &          &                  &\\ \hline
                           8       &2.22E-04         &2.98      &2.34E-06  &3.97      &1.63E-06          &4.04\\
                          16       &2.67E-05         &2.99      &1.47E-07  &3.99      &9.97E-08          &4.03\\
                          32       &3.34E-06         &3.00      &9.18E-09  &4.00      &6.20E-09          &4.01\\
                          64       &4.17E-07         &3.00      &5.74E-010 &4.00      &3.87E-010         &4.00\\
\text{Theory.rate}                 &                 &3.0       &          &4.0       &                  &4.0\\ \hline
                                   &$\ell=6$         &          &          &          &                  &\\ \hline
                           4       &1.84E-03         &3.44      &3.49E-05  &4.10      &2.99E-05          &3.83\\
                           8       &1.99E-04         &3.21      &2.11E-06  &4.05      &1.69E-06          &4.14\\
                          16       &2.38E-05         &3.07      &1.30E-07  &4.02      &1.01E-07          &4.07\\
                          32       &2.94E-06         &3.02      &8.13E-09  &4.00      &6.22E-09          &4.02\\
\text{Theory.rate}                 &                 &3.0       &          &4.0       &                  &4.0\\
\hline
\end{tabular}
\end{center}
\end{table}

Table \ref{NE:Tri-5} illustrates the numerical performance of the $P_4(T)/P_{7}(\pa T)/[P_{5}(T)]^2$ element with different values of $\gamma$. The stabilization parameter is $\rho=1$. We observe that   the theoretical rates of convergence for  $\3bare_h\3bar$, $\|e_0\|$ and $\|e_b\|_{\E_h}$  are consistent with the theoretical prediction for the case of $\gamma=-1$.  The numerical results outperform the theory for $\gamma=-10^{-5},-\frac{1}{2},1$.

\begin{table}[htbp]\centering\scriptsize
\begin{center}
\caption{Numerical errors and convergence rates for the $P_4(T)/P_{7}(\pa T)/[P_{5}(T)]^2$ element.}\label{NE:Tri-5}
\begin{tabular}{p{1.4cm}p{1.8cm}p{1.0cm}p{2cm}p{1.0cm}p{2cm}p{0.6cm}}
\hline
                         $1/h$    &$\3bare_h\3bar$        &Rate      &$\|e_0\|$   &Rate      &$\|e_b\|_{\E_h}$   &Rate\\
\hline
                                  &$\gamma=-10^{-5}$      &          &            &          &                   &\\ \hline
                          2       &1.16E-03               &4.03      &3.58E-05    &4.29      &1.67E-04           &5.78\\
                          4       &3.57E-05               &5.02      &5.62E-07    &5.99      &2.75E-06           &5.92\\
                          8       &1.08E-06               &5.04      &8.67E-09    &6.02      &4.34E-08           &5.98\\
                          16      &3.43E-08               &4.98      &1.35E-010   &6.01      &6.82E-010          &5.99\\
\text{Theory.rate}                &                       &4.0       &            &4.5       &                   &4.5\\ \hline
                                  &$\gamma=-1$            &          &            &          &                   &\\ \hline
                          2       &9.09E-05               &4.07      &1.40e-06    &5.06      &4.75E-06           &5.31\\
                          4       &9.09E-05               &4.07      &1.34E-06    &5.06      &4.75E-06           &5.31\\
                          8       &5.61E-06               &4.02      &4.33E-08    &5.01      &1.37E-07           &5.12\\
                          16      &3.50E-07               &4.00      &1.35E-09    &5.00      &4.20E-09           &5.03\\
\text{Theory.rate}                &                       &4.0       &            &5.0       &                   &5.0\\ \hline
                                  &$\gamma=-\frac{1}{2}$  &          &            &          &                   &\\ \hline
                          2       &1.33E-03               &4.55      &4.07e-05    &5.53      &3.44E-06           &5.68\\
                          4       &5.67E-05               &4.55      &8.85E-07    &5.53      &3.44E-06           &5.68\\
                          8       &2.46E-06               &4.52      &1.95E-08    &5.50      &6.90E-08           &5.64\\
                          16      &1.10E-07               &4.49      &4.34E-010   &5.49      &1.45E-09           &5.58\\
\text{Theory.rate}                &                       &4.0       &            &4.75      &                   &4.75\\ \hline
                                  &$\gamma=1$             &          &            &          &                   &\\ \hline
                          2       &8.92E-04               &4.69      &2.79e-05    &4.97      &1.53E-04           &5.91  \\
                          4       &1.52E-05               &5.87      &2.43E-07    &6.84      &2.28E-06           &6.07\\
                          8       &2.44E-07               &5.96      &1.96E-09    &6.95      &3.49E-08           &6.03\\
                          16      &1.60E-08               &3.93      &3.65E-011   &5.75      &2.29E-09           &3.93\\
\text{Theory.rate}                &                       &4.0       &            &4.0       &                   &4.0\\
\hline
\end{tabular}
\end{center}
\end{table}

Table \ref{NE:Tri-6} presents the numerical results for the $P_5(T)/P_5(\pa T)/[P_{4}(T)]^2$ element with different values of $\rho$. The stabilization parameter is $\gamma=-1$. We can see that the convergence rates are perfectly consistent with the theory prediction for $\3bare_h\3bar$ and  $\|e_0\|$, and exceed the theory prediction for $\|e_b\|$.  

\begin{table}[htbp]\centering\scriptsize
\begin{center}
\caption{Numerical errors and convergence rates for the $P_5(T)/P_5(\pa T)/[P_{4}(T)]^2$ element.}\label{NE:Tri-6}
\begin{tabular}{p{1.4cm}p{1.8cm}p{1.0cm}p{2cm}p{1.0cm}p{1.7cm}p{0.8cm}}
\hline
                         $1/h$    &$\3bare_h\3bar$  &Rate      &$\|e_0\|$   &Rate      &$\|e_b\|_{\E_h}$   &Rate\\
\hline
                                  &{ $\rho=10^{-4}$}   &          &            &          &                   &\\ \hline
                          2       &6.21E-01         &5.07      &7.19E-00    &6.08      &2.78E-00           &-8.30\\
                          4       &2.05E-02         &4.92      &1.18E-01    &5.93      &2.30E-02           &6.91\\
                          8       &6.48E-04         &4.98      &1.87E-03    &5.98      &1.83E-04           &6.97\\
                          16      &2.03E-05         &5.00      &2.93E-05    &6.00      &1.61E-06           &6.83\\
\text{Theory.rate}                &                 &5.0       &            &6.0       &                   &6.0\\ \hline
                                  &$\rho=10^{-1}$   &          &            &          &                   &\\ \hline
                          2       &1.95E-02         &5.07      &7.16E-03    &6.08      &2.78E-03           &0.26\\
                          4       &6.44E-04         &4.92      &1.18E-04    &5.93      &2.31E-05           &6.91\\
                          8       &2.04E-05         &4.98      &1.86E-06    &5.98      &1.84E-07           &6.97\\
                          16      &6.39E-07         &5.00      &2.92E-08    &6.00      &1.63E-09           &6.82\\
\text{Theory.rate}                &                 &5.0       &            &6.0       &                   &6.0\\ \hline
                                  &$\rho=1$         &          &            &          &                   &\\ \hline
                          2       &5.89E-03         &5.07      &6.88E-04    &6.08      &2.85E-04           &3.42\\
                          4       &1.94E-04         &4.92      &1.13E-05    &5.92      &2.50E-06           &6.83\\
                          8       &6.15E-06         &4.98      &1.79E-07    &5.98      &2.38E-08           &6.71\\
                          16      &1.93E-07         &5.00      &2.81E-09    &6.00      &2.88E-010          &6.37\\
\text{Theory.rate}                &                 &5.0       &            &6.0       &                   &6.0\\ \hline
                                  &$\rho=10^4$      &          &            &          &                   &\\ \hline
                          2       &4.51E-02         &5.06      &7.26E-05    &6.10      &1.37E-04           &6.09\\
                          4       &1.47E-03         &4.94      &1.18E-06    &5.95      &2.18E-06           &5.97\\
                          8       &4.65E-05         &4.98      &1.83E-08    &6.01      &3.28E-08           &6.05\\
                          16      &1.46E-06         &5.00      &2.83E-010   &6.01      &5.01E-010          &6.03\\
\text{Theory.rate}                &                 &5.0       &            &6.0       &                   &6.0\\
\hline
\end{tabular}
\end{center}
\end{table}

Table \ref{NE:Tri-7} presents some numerical results for the $P_3(T)/P_j(\pa T)/[P_\ell(T)]^2$ elements for different values of $j$ and $\ell$ when $\rho=0$ is employed in the {\rm g}WG scheme \eqref{WG-scheme}. Note that the theory established in the previous sections applies to $\rho>0$. However, we have observed from Table \ref{NE:Tri-7}  that (1) For the case of $j\geq2$ and $\ell\geq4$, the convergence rates for $\3bare_h\3bar$, $\|e_0\|$ and $\|e_b\|_{\E_h}$ are  ${\cal O}(h^{3})$, ${\cal O}(h^{4})$ and ${\cal O}(h^{4})$, respectively; (2) For the case of $j=1$ and $\ell=4$, the  convergence rates for $\3bare_h\3bar$, $\|e_0\|$ and $\|e_b\|_{\E_h}$ are ${\cal O}(h^{2})$, ${\cal O}(h^{3})$ and ${\cal O}(h^{3})$, respectively.

\begin{table}[htbp]\centering\scriptsize
\begin{center}
\caption{Numerical errors and convergence rates for the $P_3(T)/P_j(\pa T)/[P_\ell(T)]^2$ elements.}\label{NE:Tri-7}
\begin{tabular}{p{1.4cm}p{2cm}p{1cm}p{2cm}p{1cm}p{2cm}p{0.6cm}}
\hline
                         $1/h$    &$\3bare_h\3bar$  &Rate      &$\|e_0\|$  &Rate      &$\|e_b\|_{\E_h}$  &Rate\\
\hline
                                  &$j=1$,~$\ell=4$  &          &           &          &           &\\ \hline
                          16      &1.75E-03         &2.00      &1.52E-05   &3.01      &1.84E-05   &2.93\\
                          32      &4.37E-04         &2.00      &1.89E-06   &3.01      &2.34E-06   &2.97\\
                          64      &1.09E-04         &2.00      &2.36E-07   &3.00      &2.95E-07   &2.99\\
                          128     &2.74E-05         &2.00      &2.94E-08   &3.00      &3.70E-08   &2.99\\
\text{Theory.rate}                &                 &N/A       &           &N/A       &           &N/A\\ \hline
                                  &$j=2$,~$\ell=4$  &          &           &          &           &\\ \hline
                          8       &9.80E-04         &2.97      &7.52E-06   &4.03      &7.99E-06   &3.91\\
                          16      &1.23E-04         &2.99      &4.63E-07   &4.02      &5.12E-07   &3.96\\
                          32      &1.54E-05         &3.00      &2.87E-08   &4.01      &3.23E-08   &3.98\\
                          64      &1.93E-06         &3.00      &1.79E-09   &4.00      &2.03E-09   &3.99\\
\text{Theory.rate}                &                 &N/A       &           &N/A       &           &N/A\\ \hline
                 &$j=3$,~$\ell=4$  &          &           &          &           &\\ \hline
                          8       &9.07E-04         &2.97      &6.61E-06   &4.09      &1.02E-05   &3.99\\
                          16      &1.14E-04         &2.99      &4.00E-07   &4.05      &6.36E-07   &4.00\\
                          32      &1.43E-05         &3.00      &2.47E-08   &4.02      &3.98E-08   &4.00\\
                          64      &1.79E-06         &3.00      &1.53E-09   &4.01      &2.49E-09   &4.00\\
\text{Theory.rate}                &                 &N/A       &           &N/A       &           &N/A\\ \hline
                                  &$j=2$,~$\ell=5$  &          &           &          &           &\\ \hline
                          8       &1.66E-03         &2.97      &1.02E-05   &4.05      &1.25E-05   &3.93\\
                          16      &2.09E-04         &2.99      &6.26E-07   &4.03      &7.95E-07   &3.97\\
                          32      &2.62E-05         &3.00      &3.87E-08   &4.02      &5.01E-08   &3.99\\
                          64      &3.29E-06         &3.00      &2.41E-09   &4.01      &3.14E-09   &4.00\\ \hline
\text{Theory.rate}                &                 &N/A       &           &N/A       &           &N/A\\
\hline
\end{tabular}
\end{center}
\end{table}

\subsection{The {\rm gWG} elements on the uniform rectangular partition with smooth solutions}

In this section, the uniform rectangular partition is employed. 

Table \ref{NE:Squa-1} illustrates the numerical performance of the $P_k(T)/P_{k-1}(\pa T)/[P_{k-1}(T)]^2$  element  for $k=3,4$    with the exact solution $u=\cos(\pi x)\cos(\pi y)$. The stabilization parameters are given by $\rho=1$ and $\gamma=-1$. We observe from Table \ref{NE:Squa-1}  that the theoretical rates of convergence for  $\3bare_h\3bar$, $\|e_0\|$ and $\|e_b\|$ are verified by the numerical results.

\begin{table}[htbp]\centering\scriptsize
\begin{center}
\caption{Numerical errors and convergence rates for the $P_k(T)/P_{k-1}(\pa T)/[P_{k-1}(T)]^2$ element.} \label{NE:Squa-1}
\begin{tabular}{p{1.3cm}p{2cm}p{1.0cm}p{2cm}p{1.0cm}p{2cm}p{0.5cm}}
\hline
                         $1/h$    &$\3bare_h\3bar$  &Rate      &$\|e_0\|$ &Rate      &$\|e_b\|_{\E_h}$  &Rate\\ \hline
                                  &$k=3$            &          &          &          &                  &\\ \hline
                          16      &3.71E-04         &2.97      &3.32E-06  &4.03      &1.17E-05          &3.69\\
                          32      &4.67E-05         &2.99      &2.05E-07  &4.02      &7.98E-07          &3.88\\
                          64      &5.86E-06         &2.99      &1.27E-08  &4.01      &5.17E-08          &3.95\\
                          128     &7.34E-07         &3.00      &7.93E-010 &4.00      &3.29E-09          &3.98\\
\text{Theory.rate}                &                 &3.0       &          &4.0       &                  &4.0\\ \hline
                                  &$k=4$            &          &          &          &                  & \\ \hline
                          8       &1.33E-04         &3.96      &2.30E-06  &5.00      &5.14E-06          &4.29\\
                          16      &8.39E-06         &3.84      &7.15E-08  &5.01      &1.95E-07          &4.72\\
                          32      &5.27E-07         &3.99      &2.23E-09  &5.00      &6.60E-09          &4.89\\
                          64      &3.30E-08         &4.00      &6.94E-011 &5.00      &2.13E-010         &4.95\\
\text{Theory.rate}                &                 &4.0       &          &5.0       &                  &5.0\\
\hline
\end{tabular}
\end{center}
\end{table}

Table \ref{NE:Squa-2} shows some numerical results for the   $P_0(T)/P_j(\pa T)/[P_\ell(T)]^2$ element with different values of $j$ and $\ell$. The exact solution is given by $u=\cos(\pi x)\sin(\pi y)$. The stabilization parameters are $\rho=1$ and $\gamma=0$, respectively. These numerical results suggest that 1) for  the $P_0(T)/P_2(\pa T)/[P_0(T)]^2$ element and the $P_0(T)/P_3(\pa T)/[P_1(T)]^2$ element, the convergence rates   for   $\3bare_h\3bar$, $\|e_0\|$ and $\|e_b\|$  are consistent with the developed theory; 2) For the $P_0(T)/P_0(\pa T)/[P_1(T)]^2$ element, the convergence rates  for  $\3bare_h\3bar$, $\|e_0\|$ and $\|e_b\|$ outperform the theoretical prediction.

\begin{table}[htbp]\centering\scriptsize
\begin{center}
\caption{Numerical errors and convergence rates for the $P_0(T)/P_j(\pa T)/[P_\ell(T)]^2$ elements.} \label{NE:Squa-2}
\begin{tabular}{p{1.0cm}p{2cm}p{1cm}p{2cm}p{1cm}p{2cm}p{0.5cm}}
\hline
                           $1/h$    &$\3bare_h\3bar$  &Rate      &$\|e_0\|$ &Rate      &$\|e_b\|_{\E_h}$ &Rate\\
\hline
                                    &$j=2$,~$\ell=0$  &          &          &          &                 &\\ \hline
                            16      &1.13E-00         &0.48      &7.65E-02  &0.99      &1.23E-01         &0.87\\
                            32      &8.06E-01         &0.49      &3.81E-02  &1.01      &6.56E-02         &0.90\\
                            64      &5.72E-01         &0.50      &1.90E-02  &1.01      &3.42E-02         &0.94\\
                            128     &4.05E-01         &0.50      &9.46E-03  &1.00      &1.75E-02         &0.97\\
\text{Theory.rate}                  &                 &0.5       &          &1.0       &                 &1.0\\ \hline
                                    &$j=3$,~$\ell=1$  &          &          &          &                 &\\ \hline
                            16      &2.00E-01         &0.55      &3.33E-03  &1.06      &6.30E-02         &0.90\\
                            32      &1.42E-01         &0.49      &1.62E-03  &1.04      &3.40E-02         &0.89\\
                            64      &1.02E-01         &0.47      &8.00E-04  &1.02      &1.81E-02         &0.91\\
                            128     &7.39e-02         &0.47      &3.97E-04  &1.01      &9.50E-03         &0.93\\
\text{Theory.rate}                  &                 &0.5       &          &1.0       &                 &1.0\\ \hline
                                    &$j=0$,~$\ell=1$  &          &          &          &                 &\\ \hline
                            16      &3.44E-02         &1.09      &3.33E-03  &1.06      &1.65E-02         &1.09\\
                            32      &1.66E-02         &1.05      &1.62E-03  &1.04      &7.97E-03         &1.05\\
                            64      &8.15E-03         &1.03      &8.00E-04  &1.02      &3.92E-03         &1.02\\
                           128      &4.04E-03         &1.01      &3.97E-04  &1.01      &1.94E-03         &1.01\\
\text{Theory.rate}                  &                 &0.0       &          &0.0       &                 &0.0\\
\hline
\end{tabular}
\end{center}
\end{table}

Table \ref{NE:Squa-3} demonstrates the numerical performance of the $P_5(T)/P_j(\pa T)/[P_\ell(T)]^2$ element  when   $j<\ell$. We choose  $\rho=1$ and $\gamma=-1$. The exact solution is $u=x^2\cos(\pi y)$. We observe from Table \ref{NE:Squa-3}  that (1) the convergence rates for  $\3bare_h\3bar$, $\|e_0\|$ and $\|e_b\|$ consist with the theoretical convergence rates when the cases of $(j,\ell)=(0, 5)$, $(j,\ell)=(1, 4)$    are applied; (2) for the  case of $(j,\ell)=(2, 3)$,  the convergence rates for  $\3bare_h\3bar$, $\|e_0\|$  are  consistent with the theoretical convergence rates while the convergence rate for $\|e_b\|_{\E_h}$ seems to exceed the theoretical convergence rate of ${\cal O}(h^4)$. 

\begin{table}[htbp]\centering\scriptsize
\begin{center}
\caption{Numerical errors and convergence rates for the $P_5(T)/P_j(\pa T)/[P_\ell(T)]^2$ elements.}\label{NE:Squa-3}
\begin{tabular}{p{1.5cm}p{2cm}p{1cm}p{2cm}p{1cm}p{2cm}p{0.5cm}}
\hline
                           $1/h$    &$\3bare_h\3bar$  &Rate      &$\|e_0\|$  &Rate      &$\|e_b\|_{\E_h}$  &Rate\\
\hline
                                    &$j=0$,~$\ell=5$  &          &           &          &                  &\\ \hline
                            8       &1.11E-01         &0.98      &1.59E-03   &1.97      &1.42E-03          &1.94\\
                            16      &5.56E-02         &0.99      &4.00E-04   &1.99      &3.57E-04          &1.99\\
                            32      &2.78E-02         &1.00      &1.00E-04   &2.00      &8.95E-05          &2.00\\
                            64      &1.39E-02         &1.00      &2.50E-05   &2.00      &2.24E-05          &2.00\\
\text{Theory.rate}                  &                 &1.0       &           &2.0       &                  &2.0\\ \hline
                                    &$j=1$,~$\ell=4$  &          &           &          &                  &\\ \hline
                            8       &2.99E-03         &1.91      &2.51E-05   &2.78      &1.25E-04          &2.51\\
                            16      &7.70E-04         &1.96      &3.38E-06   &2.89      &1.77E-05          &2.83\\
                            32      &1.95E-04         &1.98      &4.37E-07   &2.95      &2.32E-06          &2.93\\
                            64      &4.92E-05         &1.99      &5.55E-08   &2.98      &2.97E-07          &2.97\\
\text{Theory.rate}                  &                 &2.0       &           &3.0       &                  &3.0\\ \hline
                                    &$j=2$,~$\ell=3$  &          &           &          &                  &\\ \hline
                            8       &1.27E-04         &3.12      &1.74E-06   &4.19      &8.24E-07          &4.49\\
                            16      &1.55E-05         &3.04      &1.04E-07   &4.06      &2.96E-08          &4.80\\
                            32      &1.92E-06         &3.01      &6.45E-09   &4.01      &9.81E-010         &4.92\\
                            64      &2.40E-07         &3.00      &4.02E-010  &4.00      &3.16E-011         &4.96\\
\text{Theory.rate}                  &                 &3.0       &           &4.0       &                  &4.0\\
\hline
\end{tabular}
\end{center}
\end{table}

Table \ref{NE:Squa-5} reports the errors and convergence rates for the $P_3(T)/P_j(\pa T)/[P_\ell(T)]^2$   element when $j\geq\ell$. The stabilization parameters are $\rho=1$ and $\gamma=-1$. The exact solution is given by $u=x^2\cos(\pi y)$. Table \ref{NE:Squa-5} implies that the convergence rates for $\3bare_h\3bar$, $\|e_0\|$ and $\|e_b\|$   are  consistent with the theoretical convergence rates when the cases of $(j,\ell)=(1, 1)$, $(j,\ell)=(3, 3)$, $(j,\ell)=(4, 3)$, and $(j,\ell)=(2, 1)$    are applied.

\begin{table}[htbp]\centering\scriptsize
\begin{center}
\caption{Numerical errors and convergence rates for the $P_3(T)/P_j(\pa T)/[P_\ell(T)]^2$ elements.}\label{NE:Squa-5}
\begin{tabular}{p{1.3cm}p{2cm}p{1cm}p{2cm}p{1cm}p{2cm}p{0.5cm}}
\hline
                           $1/h$    &$\3bare_h\3bar$  &Rate       &$\|e_0\|$   &Rate      &$\|e_b\|_{\E_h}$  &Rate\\
\hline
                                    &$j=1$,~$\ell=1$  &           &            &          &                  &\\ \hline
                            16      &5.87E-03         &1.91       &6.15E-05    &2.90      &4.70E-04          &2.76\\
                            32      &1.51E-03         &1.95       &7.99E-06    &2.94      &6.32E-05          &2.89\\
                            64      &3.84E-04         &1.98       &1.02E-06    &2.97      &8.18E-06          &2.95\\
                            128     &9.67E-05         &1.99       &1.29E-07    &2.99      &1.04E-06          &2.98\\
\text{Theory.rate}                  &                 &2.0        &            &3.0       &                  &3.0\\ \hline
                                    &$j=3$,~$\ell=3$  &           &            &          &                  &\\ \hline
                            8       &5.93E-05         &3.54       &6.31E-07    &4.91      &4.95E-06          &4.57\\
                            16      &6.09E-06         &3.28       &2.23E-08    &4.83      &2.16E-07          &4.52\\
                            32      &7.12E-07         &3.10       &9.29E-010   &4.58      &1.11E-08          &4.28\\
                            64      &8.75E-08         &3.02       &4.79E-011   &4.28      &6.47E-010         &4.10\\
\text{Theory.rate}                  &                 &3.0        &            &4.0       &                  &4.0\\ \hline
                                    &$j=4$,~$\ell=3$  &           &            &          &                  &\\ \hline
                            8       &6.27E-05         &3.49       &6.31E-07    &4.91      &5.94E-06          &4.41\\
                            16      &6.66E-06         &3.23       &2.23E-08    &4.83      &3.06E-07          &4.28\\
                            32      &7.93E-07         &3.07       &9.29E-010   &4.58      &1.78E-08          &4.10\\
                            64      &9.80E-08         &3.02       &4.79E-011   &4.28      &1.09E-09          &4.02\\
\text{Theory.rate}                  &                 &3.0        &            &4.0       &                  &4.0\\ \hline
                                    &$j=2$,~$\ell=1$  &           &            &          &                  &\\ \hline
                            8       &2.19E-02         &1.82       &4.58E-04    &2.81      &3.20E-03          &2.51\\
                            16      &5.86E-03         &1.90       &6.15E-05    &2.90      &4.71E-04          &2.76\\
                            32      &1.51E-03         &1.95       &7.99E-06    &2.94      &6.35E-05          &2.89\\
                            64      &3.84E-04         &1.98       &1.02E-06    &2.97      &8.21E-06          &2.95\\
\text{Theory.rate}                  &                 &2.0        &            &3.0       &                  &3.0\\
\hline
\end{tabular}
\end{center}
\end{table}

Table \ref{NE:Squa-6} presents some numerical results for the $P_3(T)/P_2(\pa T)/[P_3(T)]^2$ element for $\rho=1$ and different stabilization parameter $\gamma$. The exact solution is $u=x^2\cos(\pi y)$. We observe from Table \ref{NE:Squa-6} that (1) for  $\gamma=-1$, the convergence rates for $\3bare_h\3bar$ and $\|e_b\|$ are in great consistency with the theoretical convergence rates; while  the convergence rate for $\|e_0\|$ is higher than the theoretical prediction; (2) for  $\gamma=-3, 0, 1$, the convergence rates for $\3bare_h\3bar$, $\|e_0\|$ and $\|e_b\|$ are higher than what our theory predicts.

\begin{table}[htbp]\centering\scriptsize
\begin{center}
\caption{Numerical errors and convergence rates for the $P_3(T)/P_2(\pa T)/[P_3(T)]^2$ element.}\label{NE:Squa-6}
\begin{tabular}{p{1.4cm}p{2cm}p{1.0cm}p{2cm}p{1.0cm}p{2cm}p{0.5cm}}
\hline
                         $1/h$    &$\3bare_h\3bar$  &Rate      &$\|e_0\|$  &Rate      &$\|e_b\|_{\E_h}$   &Rate\\
\hline
                                  &$\gamma=-3$      &          &           &          &                   &\\ \hline
                          8       &1.26E-03         &1.79      &2.77E-06   &3.45      &2.13E-05           &3.55\\
                          16      &3.28E-04         &1.94      &2.62E-07   &3.40      &1.50E-06           &3.83\\
                          32      &8.27E-05         &1.99      &2.93E-08   &3.16      &1.17E-07           &3.69\\
                          64      &2.07E-05         &2.00      &3.56E-09   &3.04      &1.14E-08           &3.35\\
\text{Theory.rate}                &                 &2.0       &           &2.0       &                   &2.0\\ \hline
                                  &$\gamma=-1$      &          &           &          &                   &\\ \hline
                          8       &5.41E-05         &3.48      &5.68E-07   &4.97      &2.68E-06           &4.00\\
                          16      &5.87E-06         &3.20      &2.01E-08   &4.82      &1.63E-07           &4.04\\
                          32      &7.04E-07         &3.06      &8.72E-010  &4.52      &1.01E-08           &4.01\\
                          64      &8.73E-08         &3.01      &4.68E-011  &4.22      &6.30E-010          &4.00\\
\text{Theory.rate}                &                 &3.0       &           &4.0       &                   &4.0\\ \hline
                                  &$\gamma=0$       &          &           &          &                   &\\ \hline
                          8       &1.03E-04         &3.47      &6.00E-06   &3.97      &6.88E-07           &4.62\\
                          16      &9.17E-06         &3.49      &3.79E-07   &3.99      &2.38E-08           &4.86\\
                          32      &8.12E-07         &3.50      &2.38E-08   &3.99      &7.73E-010          &4.94\\
                          64      &7.19E-08         &3.50      &1.49E-09   &4.00      &2.46E-011          &4.97\\
\text{Theory.rate}                &                 &2.5       &           &3.0       &                   &3.0\\ \hline
                                  &$\gamma=1$       &          &           &          &                   &\\ \hline
                          8       &3.24E-04         &2.98      &5.95E-05   &2.98      &5.95E-07           &4.56\\
                          16      &4.07E-05         &3.00      &7.47E-06   &2.99      &2.12E-08           &4.81\\
                          32      &5.09E-06         &3.00      &9.35E-07   &3.00      &7.02E-010          &4.92\\
                          64      &6.36E-07         &3.00      &1.17E-07   &3.00      &2.26E-011          &4.96\\
\text{Theory.rate}                &                 &2.0       &           &2.0       &                   &2.0\\
\hline
\end{tabular}
\end{center}
\end{table}

 Table \ref{SFIWGG} illustrates the numerical performance of the {\rm g}WG scheme \eqref{WG-scheme} when $\rho=0$ and the $P_2(T)/P_j(\pa T)/[P_\ell(T)]^2$ element are applied. The exact solution is $u=x^2\cos(\pi y)$. We observe from   Table \ref{SFIWGG} that (1) for the cases of $(j,\ell)=(1, 4)$, $(j,\ell)=(1, 5)$ and $(j,\ell)=(2, 4)$, the convergence rates for $\3bare_h\3bar$, $\|e_0\|$ and $\|e_b\|_{\E_h}$ are in an order of ${\cal O}(h^2)$, ${\cal O}(h^3)$ and ${\cal O}(h^3)$, respectively; (2) for the case   of $(j,\ell)=(1, 3)$, the convergence rates for $\3bare_h\3bar$, $\|e_0\|$ and $\|e_b\|_{\E_h}$ are in an order of ${\cal O}(h^3)$, ${\cal O}(h^4)$ and ${\cal O}(h^4)$, respectively; and (3) for the case   of $(j,\ell)=(0, 4)$, the convergence rates for $\3bare_h\3bar$, $\|e_0\|$ and $\|e_b\|_{\E_h}$ are in an order of ${\cal O}(h)$, ${\cal O}(h^2)$ and ${\cal O}(h^2)$, respectively. Note that our theory established in this paper does not apply to  the case of $\rho=0$.  Readers are encouraged to draw their own conclusions.

\begin{table}[htbp]\centering\scriptsize
\begin{center}
\caption{Numerical errors and convergence rates for the $P_2(T)/P_j(\pa T)/[P_\ell(T)]^2$ elements.}\label{SFIWGG}
\begin{tabular}{p{1.5cm}p{2cm}p{1cm}p{2cm}p{1cm}p{2cm}p{0.5cm}}
\hline
                         $1/h$    &$\3bare_h\3bar$  &Rate      &$\|e_0\|$  &Rate      &$\|e_b\|_{\E_h}$ &Rate\\
\hline
                                  &$j=1$,~$\ell=3$  &          &           &          &                 &\\ \hline
                          16      &2.04E-05         &2.99      &2.18E-07   &3.99      &5.03E-05         &3.99\\
                          32      &2.55E-06         &3.00      &1.37E-08   &4.00      &3.14E-06         &4.00\\
                          64      &3.19E-07         &3.00      &8.55E-010  &4.00      &1.97E-07         &4.00\\
                          128     &3.99E-08         &3.00      &5.38E-011  &3.99      &1.24E-010        &4.00\\
\text{Theory.rate}                &                 &N/A       &           &N/A       &                 &N/A\\ \hline
                                  &$j=1$,~$\ell=4$  &          &           &          &                 &\\ \hline
                          16      &2.88E-03         &1.95      &5.25E-06   &3.01      &6.36E-05         &2.97\\
                          32      &7.33E-04         &1.98      &6.55E-07   &3.00      &8.02E-06         &2.99\\
                          64      &1.85E-04         &1.99      &8.18E-08   &3.00      &1.01E-06         &2.99\\
                          128     &4.64E-05         &1.99      &1.02E-08   &3.00      &1.26E-07         &3.00\\
\text{Theory.rate}                &                 &N/A       &           &N/A       &                 &N/A\\ \hline
                                  &$j=1$,~$\ell=5$  &          &           &          &                 &\\ \hline
                          16      &2.84E-03         &1.92      &5.15E-06   &2.99      &6.07E-05         &2.90\\
                          32      &7.27E-04         &1.96      &6.48E-07   &2.99      &7.84E-06         &2.95\\
                          64      &1.84E-04         &1.98      &8.14E-08   &2.99      &9.96E-07         &2.98\\
                          128     &4.63E-05         &1.99      &1.02E-08   &3.00      &1.25E-07         &2.99\\
\text{Theory.rate}                &                 &N/A       &           &N/A       &                 &N/A\\ \hline
                                  &$j=2$,~$\ell=4$  &          &           &          &                 &\\ \hline
                          16      &2.89E-03         &1.95      &5.44E-06   &3.13      &6.41E-05         &3.00\\
                          32      &7.33E-04         &1.98      &6.61E-07   &3.04      &8.04E-06         &3.00\\
                          64      &1.85E-04         &1.99      &8.20E-08   &3.01      &1.01E-06         &3.00\\
                          128     &4.64E-05         &1.99      &1.02E-08   &3.00      &1.26E-07         &3.00\\
\text{Theory.rate}                &                 &N/A       &           &N/A       &                 &N/A\\ \hline
                                  &$j=0$,~$\ell=4$  &          &           &          &                 &\\ \hline
                          16      &5.55E-02         &0.99      &4.00E-04   &1.99      &3.57E-04         &1.99\\
                          32      &2.78E-02         &1.00      &1.00E-04   &2.00      &8.95E-05         &2.00\\
                          64      &1.39E-02         &1.00      &2.50E-05   &2.00      &2.24E-05         &2.00\\
                          128     &6.95E-03         &1.00      &6.26E-06   &2.00      &5.59E-06         &2.00\\
\text{Theory.rate}                &                 &N/A       &           &N/A       &                 &N/A\\
\hline
\end{tabular}
\end{center}
\end{table}

Table \ref{NE:Squa-7} shows the numerical results for the $P_3(T)/P_2(\pa T)/[P_4(T)]^2$ element with different values of $\rho$ and $\gamma$. We observe from Table \ref{NE:Squa-7} that (1) for the cases of $(\rho, \gamma)=(1, -1)$ and $(\rho, \gamma)=(10^2, -1)$, the rates of convergence for $\3bare_h\3bar$, $\|e_0\|$ and $\|e_b\|_{\E_h}$  are consistent with the theoretical rates of convergence; (2) for the case of $(\rho, \gamma)=(10^{-12},-1)$, the convergence rates for $\3bare_h\3bar$, $\|e_0\|$ and $\|e_b\|_{\E_h}$ seem to outperform the theoretical prediction; (3) for the case of $\rho=0$ where the stabilizer is $0$ for any value of $\gamma$, the convergence rates for $\3bare_h\3bar$, $\|e_0\|$ and $\|e_b\|_{\E_h}$ seem to be in an order of ${\cal O}(h^4)$, ${\cal O}(h^5)$ and ${\cal O}(h^5)$, respectively. Again, our established theory does not cover the case of $\rho=0$. 

\begin{table}[htbp]\centering\scriptsize
\begin{center}
\caption{Numerical errors and convergence rates for the $P_3(T)/P_2(\pa T)/[P_4(T)]^2$ element.}\label{NE:Squa-7}
\begin{tabular}{p{1.5cm}p{2.5cm}p{0.6cm}p{2cm}p{0.6cm}p{2cm}p{0.5cm}}
\hline
                         $1/h$    &$\3bare_h\3bar$&Rate      &$\|e_0\|$   &Rate      &$\|e_b\|_{\E_h}$   &Rate\\
\hline
                                  &$\rho=0$, $\forall \gamma$       &          &            &          &                   &\\ \hline
                          4       &1.95E-04       &3.97      &4.57E-06    &4.79      &1.15E-05           &5.09\\
                          8       &1.24E-05       &3.98      &1.54E-07    &4.89      &3.68E-07           &4.96\\
                          16      &7.82E-07       &3.99      &4.98E-09    &4.95      &1.19E-08           &4.95\\
                          32      &4.91E-08       &3.99      &1.58E-010   &4.98      &3.79E-010          &4.97\\
\text{Theory.rate}                &               &N/A       &            &N/A       &                   &N/A\\ \hline
                                  &$\rho=10^{-12}$,~$\gamma=-1$&&         &          &                   &\\ \hline
                          4       &1.95E-04       &3.97      &4.57E-06    &4.79      &1.15E-05           &5.09\\
                          8       &1.24E-05       &3.98      &1.54E-07    &4.89      &3.68E-07           &4.96\\
                          16      &7.82E-07       &3.99      &4.98E-09    &4.95      &1.19E-08           &4.95\\
                          32      &4.91E-08       &3.99      &1.58E-010   &4.98      &3.79E-010          &4.97\\
\text{Theory.rate}                &               &3.0       &            &4.0       &                   &4.0\\ \hline
                                  &$\rho=1$,~$\gamma=-1$&&                &          &                   &\\ \hline
                          4       &1.13E-03       &3.00      &1.16E-05    &4.30      &1.16E-04           &3.69\\
                          8       &1.43E-04       &2.99      &5.88E-07    &4.15      &7.69E-06           &3.91\\
                          16      &1.79E-05       &2.99      &3.31E-08    &4.07      &4.90E-07           &3.97\\
                          32      &2.25E-06       &3.00      &1.97E-09    &4.03      &3.08E-08           &3.99\\
\text{Theory.rate}                &               &3.0       &            &4.0       &                   &4.0\\ \hline
                                  &$\rho=10^{2}$,~$\gamma=-1$&&           &          &                   &\\ \hline
                          16      &5.06E-04       &3.00      &3.66E-07    &4.20      &3.14E-06           &4.08\\
                          32      &6.32E-05       &3.00      &2.12E-08    &4.11      &1.91E-07           &4.04\\
                          64      &7.91E-06       &3.00      &1.28E-09    &4.05      &1.18E-08           &4.02\\
                          128     &9.88E-07       &3.00      &7.99E-011   &4.01      &7.37E-010          &4.00\\
\text{Theory.rate}                &               &3.0       &            &4.0       &                   &4.0\\ \hline
\end{tabular}
\end{center}
\end{table}

\subsection{Solutions with low regularity}

In this section,  the exact solution is given by $u=x(x-1)y(y-1)(x^2+y^2)^{(-2+\alpha)/2}$ where $\alpha\in (0, 1]$. It is easy to verify $u\in H^{1+\alpha-\tau}(\O)$ for an arbitrary small $\tau>0$. We should point out that our theory is not developed  for the low regularity solution.

Table \ref{low-1} illustrates the numerical performance of the $P_k(T)/P_k(\pa T)/[P_{k-1}(T)]^2$ element for different  $k$ on the uniform triangular partition.  We choose $\alpha=\frac{1}{2}$. This implies the exact solution $u\in H^{1.5-\tau}(\O)$ does not satisfy the required regularity assumption.   The stabilization parameters are $\rho=1$ and $\gamma=-1$. We can observe that the  convergence rates for $\3bare_h\3bar$, $\|e_0\|$ and $\|e_b\|_{\E_h}$
seem to  be in an order of ${\cal O}(h^{0.5})$, ${\cal O}(h^{1.5})$ and ${\cal O}(h^{1.5})$, respectively.  

\begin{table}[htbp]\centering\scriptsize
\begin{center}
\caption{Numerical errors and convergence rates for the $P_k(T)/P_k(\pa T)/[P_{k-1}(T)]^2$ element.} \label{low-1}
\begin{tabular}{p{1.5cm}p{1.9cm}p{1.0cm}p{1.9cm}p{1.0cm}p{1.9cm}p{0.5cm}}
\hline
                         $1/h$    &$\3bare_h\3bar$&Rate      &$\|e_0\|$  &Rate      &$\|e_b\|_{\E_h}$ &Rate\\
\hline
                                  &$k=1$          &          &           &          &                 &\\ \hline
                          8       &8.04E-01       &0.47      &4.32E-02   &1.48      &3.95E-02         &1.54\\
                          16      &5.75E-01       &0.48      &1.54E-02   &1.49      &1.38E-02         &1.52\\
                          32      &4.10E-01       &0.49      &5.49E-03   &1.49      &4.83E-03         &1.51\\
                          64      &2.91E-01       &0.49      &1.95E-03   &1.49      &1.70E-03         &1.51\\
\text{Theory.rate}                &               &N/A       &           &N/A       &                 &N/A\\ \hline
                                  &$k=2$          &          &           &          &                 &\\ \hline
                          8       &6.42E-01       &0.50      &2.40E-02   &1.50      &2.43E-02         &1.54\\
                          16      &4.55E-01       &0.50      &8.47E-03   &1.50      &8.48E-03         &1.52\\
                          32      &3.22E-01       &0.50      &2.99E-03   &1.50      &2.98E-03         &1.51\\
                          64      &2.28E-01       &0.50      &1.06E-03   &1.50      &1.05E-03         &1.50\\
\text{Theory.rate}                &               &N/A       &           &N/A       &                 &N/A\\ \hline
                                  &$k=3$          &          &           &          &                 &\\ \hline
                          8       &5.16E-01       &0.49      &1.55E-02   &1.50      &1.56E-02         &1.53\\
                          16      &3.66E-01       &0.50      &5.48E-03   &1.50      &5.46E-03         &1.51\\
                          32      &2.59E-01       &0.50      &1.94E-03   &1.50      &1.92E-03         &1.51\\
                          64      &1.83E-01       &0.50      &6.86E-04   &1.50      &6.78E-04         &1.50\\
\text{Theory.rate}                &               &N/A       &           &N/A       &                 &N/A\\ \hline
                                  &$k=4$          &          &           &          &                 &\\ \hline
                          8       &3.93E-01       &0.49      &1.01E-02   &1.50      &9.24E-03         &1.52\\
                          16      &2.78E-01       &0.50      &3.57E-03   &1.50      &3.24E-03         &1.51\\
                          32      &1.97E-01       &0.50      &1.26E-03   &1.50      &1.14E-03         &1.51\\
                          64      &1.39E-01       &0.50      &4.47E-04   &1.50      &4.03E-04         &1.50\\
\text{Theory.rate}                &               &N/A       &           &N/A       &                 &N/A\\
\hline
\end{tabular}
\end{center}
\end{table}

Table \ref{low-2} presents some numerical results for the $P_1(T)/P_{0}(\pa T)/[P_{0}(T)]^2$ element  on the uniform rectangular partition. The stabilization parameters are given by $\rho=1$ and $\gamma=-1$. We can observe from Table \ref{low-2} that  (1) for  $\alpha=1$, the convergence rates for $\3bare_h\3bar$, $\|e_0\|$ and $\|e_b\|_{\E_h}$ seem to be in an order of ${\cal O}(h^{0.9})$, ${\cal O}(h^{1.9})$ and ${\cal O}(h^{1.9})$, respectively; (2) for   $\alpha=\frac{1}{2},\frac{1}{8}$ and $\frac{1}{32}$, the convergence rates for $\3bare_h\3bar$, $\|e_0\|$ and $\|e_b\|_{\E_h}$ seem to be in an order of ${\cal O}(h^{\alpha})$, ${\cal O}(h^{1+\alpha})$ and ${\cal O}(h^{1+\alpha})$, respectively.

\begin{table}[htbp]\centering\scriptsize
\begin{center}
\caption{Numerical errors and convergence rates for the $P_1(T)/P_{0}(\pa T)/[P_{0}(T)]^2$ element.
}\label{low-2}
\begin{tabular}{p{1.2cm}p{2cm}p{1.0cm}p{2cm}p{1.0cm}p{2cm}p{0.5cm}}
\hline
                         $1/h$    &$\3bare_h\3bar$&Rate      &$\|e_0\|$  &Rate      &$\|e_b\|_{\E_h}$  &Rate\\ \hline
                                  &$\alpha=1$     &          &           &          &                  &\\ \hline
                          16      &6.41E-02       &0.87      &1.33E-03   &1.83      &2.60E-03          &1.67\\
                          32      &3.46E-02       &0.89      &3.65E-04   &1.86      &7.53E-04          &1.78\\
                          64      &1.85E-02       &0.90      &9.89E-05   &1.88      &2.10E-04          &1.84\\
                          128     &9.82E-03       &0.91      &2.65E-05   &1.90      &5.75E-05          &1.87\\
\text{Theory.rate}                &               &N/A       &           &N/A       &                  &N/A\\  \hline
                                  &$\alpha=1/2$   &          &           &          &                  &\\ \hline
                          16      &5.89E-01       &0.49      &9.84E-03   &1.49      &9.15E-03          &1.41\\
                          32      &4.19E-01       &0.49      &3.50E-03   &1.49      &3.34E-03          &1.45\\
                          64      &2.97E-01       &0.50      &1.24E-03   &1.49      &1.20E-03          &1.48\\
                          128     &2.10E-01       &0.50      &4.40E-04   &1.50      &4.27E-04          &1.49\\
\text{Theory.rate}                &               &N/A       &           &N/A       &                  &N/A\\ \hline
                                  &$\alpha=1/8$   &          &           &          &                  &\\ \hline
                          16      &5.01E-00       &0.12      &8.17E-02   &1.12      &9.43E-02          &1.12\\
                          32      &4.60E-00       &0.12      &3.75E-02   &1.12      &4.33E-02          &1.12\\
                          64      &4.22E-00       &0.12      &1.72E-02   &1.12      &1.99E-02          &1.12\\
                          128     &3.87E-00       &0.12      &7.90E-03   &1.12      &9.11E-03          &1.12\\
\text{Theory.rate}                &               &N/A       &           &N/A       &                  &N/A\\ \hline
                                  &$\alpha=1/32$  &          &           &          &                  &\\ \hline
                          16      &8.80E-00       &0.03      &1.44E-01   &1.03      &1.76E-01          &1.03\\
                          32      &8.63E-00       &0.03      &7.04E-02   &1.03      &8.62E-02          &1.03\\
                          64      &8.45E-00       &0.03      &3.45E-02   &1.03      &4.22E-02          &1.03\\
                          128     &8.27E-00       &0.03      &1.69E-02   &1.03      &2.06E-02          &1.03\\
\text{Theory.rate}                &               &N/A       &           &N/A       &                  &N/A\\
\hline
\end{tabular}
\end{center}
\end{table}

Table \ref{low-3} presents some computational results for the $P_1(T)/P_{0}(\pa T)/[P_{2}(T)]^2$ element  on the uniform triangular partition. We take $\rho=0$ for which there is no theroy available.  We can observe from Table \ref{low-3} that (1) for  $\alpha=1$, the convergence rates for $\3bare_h\3bar$, $\|e_0\|$ and $\|e_b\|_{\E_h}$ seem to be in an order of ${\cal O}(h^{0.9})$, ${\cal O}(h^{1.9})$ and ${\cal O}(h^{1.9})$, respectively; (2) for   $\alpha=\frac{2}{5},\frac{1}{8}$ and $\frac{1}{32}$, the convergence rates for $\3bare_h\3bar$, $\|e_0\|$ and $\|e_b\|_{\E_h}$ seem to be in an order of ${\cal O}(h^{\alpha})$, ${\cal O}(h^{1+\alpha})$ and ${\cal O}(h^{1+\alpha})$, respectively.

\begin{table}[htbp]\centering\scriptsize
\begin{center}
\caption{Numerical errors and convergence rates for the $P_1(T)/P_{0}(\pa T)/[P_{2}(T)]^2$ element.}\label{low-3}
\begin{tabular}{p{1.5cm}p{2cm}p{1.0cm}p{2cm}p{1.0cm}p{2cm}p{0.5cm}}
\hline
                         $1/h$    &$\3bare_h\3bar$&Rate      &$\|e_0\|$  &Rate      &$\|e_b\|_{\E_h}$   &Rate\\ \hline
                                  &$\alpha=1$     &          &           &          &                   &\\ \hline
                          32      &1.44E-02       &0.83      &1.29E-04   &1.77      &2.20E-04           &1.79\\
                          64      &7.92E-03       &0.86      &3.63E-05   &1.82      &6.14E-05           &1.84\\
                          128     &4.30E-03       &0.88      &1.00E-05   &1.86      &1.68E-05           &1.87\\
                          256     &2.31E-03       &0.90      &2.72E-06   &1.88      &4.51E-06           &1.89\\
\text{Theory.rate}                &               &N/A       &           &N/A       &                   &N/A\\  \hline
                                  &$\alpha=2/5$   &          &           &          &                   &\\ \hline
                          16      &2.33E-01       &0.41      &3.40E-03   &1.41      &4.32E-03           &1.41\\
                          32      &1.76E-01       &0.41      &1.28E-03   &1.41      &1.63E-03           &1.41\\
                          64      &1.33E-01       &0.40      &4.84E-04   &1.40      &6.16E-04           &1.40\\
                          128     &1.01E-01       &0.40      &1.83E-04   &1.40      &2.33E-04           &1.40\\
\text{Theory.rate}                &               &N/A       &           &N/A       &                   &N/A\\  \hline
                                  &$\alpha=1/8$   &          &           &          &                   &\\ \hline
                          16      &1.07E-00       &0.13      &1.68E-02   &1.13      &2.25E-02           &1.13\\
                          32      &9.76E-01       &0.13      &7.68E-03   &1.13      &1.03E-02           &1.13\\
                          64      &8.95E-01       &0.13      &3.51E-03   &1.13      &4.71E-03           &1.13\\
                          128     &8.20E-01       &0.13      &1.61E-03   &1.13      &2.52E-03           &1.13\\
\text{Theory.rate}                &               &N/A       &           &N/A       &                   &N/A\\  \hline
                                  &$\alpha=1/32$  &          &           &          &                   &\\ \hline
                          16      &1.78E-00       &0.03      &2.83E-02   &1.03      &3.85E-02           &1.03\\
                          32      &1.74E-00       &0.03      &1.39E-02   &1.03      &1.88E-02           &1.03\\
                          64      &1.70E-00       &0.03      &6.77E-03   &1.03      &9.20E-03           &1.03\\
                          128     &1.67E-00       &0.03      &3.31E-03   &1.03      &4.50E-03           &1.03\\
\text{Theory.rate}                &               &N/A       &           &N/A       &                   &N/A\\
\hline
\end{tabular}
\end{center}
\end{table}

Tables \ref{low-4} demonstrates the numerical performance of the $P_k(T)/P_{k-1}(\pa T)/[P_{k+1}(T)]^2$ element on the uniform rectangular partition. We take $\rho=0$ and $\alpha=\frac{2}{5}$. The exact solution has the regularity of $H^{7/5-\tau}(\O)$ for an arbitrary small $\tau>0$. These numerical results indicate that for $k=1, 2, 3$, the convergence rates for $\3bare_h\3bar$, $\|e_0\|$ and $\|e_b\|_{\E_h}$ seem to be in an order of ${\cal O}(h^{0.4})$, ${\cal O}(h^{1.4})$ and ${\cal O}(h^{1.4})$, respectively.

\begin{table}[htbp]\centering\scriptsize
\begin{center}
\caption{Numerical errors and convergence rates for the $P_k(T)/P_{k-1}(\pa T)/[P_{k+1}(T)]^2$ element.}\label{low-4}
\begin{tabular}{p{1.5cm}p{2cm}p{1.0cm}p{2cm}p{1.0cm}p{2cm}p{0.5cm}}
\hline
                         $1/h$    &$\3bare_h\3bar$&Rate      &$\|e_0\|$  &Rate      &$\|e_b\|_{\E_h}$   &Rate\\ \hline
                                  &$k=1$          &          &           &          &                   &\\ \hline
                          8       &3.27E-01       &0.41      &6.35E-03   &1.39      &3.10E-02           &1.39\\
                          16      &2.47E-01       &0.40      &2.40E-03   &1.40      &1.17E-02           &1.40\\
                          32      &1.87E-01       &0.40      &9.06E-04   &1.41      &4.42E-03           &1.41\\
                          64      &1.42E-01       &0.40      &3.43E-04   &1.40      &1.67E-03           &1.40\\
\text{Theory.rate}                &               &N/A       &           &N/A       &                   &N/A\\  \hline
                                  &$k=2$          &          &           &          &                   &\\ \hline
                          8       &2.99E-01       &0.40      &3.56E-03   &1.40      &1.05E-02           &1.39\\
                          16      &2.26E-01       &0.40      &1.35E-03   &1.40      &4.00E-03           &1.39\\
                          32      &1.71E-01       &0.40      &5.11E-04   &1.40      &1.52E-03           &1.40\\
                          64      &1.30E-01       &0.40      &1.94E-04   &1.40      &5.77E-04           &1.40\\
\text{Theory.rate}                &               &N/A       &           &N/A       &                   &N/A\\  \hline
                                  &$k=3$          &          &           &          &                   &\\ \hline
                          8       &2.98E-01       &0.40      &2.45E-03   &1.40      &5.38E-03           &1.39\\
                          16      &2.26E-01       &0.40      &9.30E-04   &1.40      &2.05E-03           &1.39\\
                          32      &1.71E-01       &0.40      &3.53E-04   &1.40      &7.78E-04           &1.40\\
                          64      &1.30E-01       &0.40      &1.34E-04   &1.40      &2.95E-04           &1.40\\
\text{Theory.rate}                &               &N/A       &           &N/A       &                   &N/A\\
\hline
\end{tabular}
\end{center}
\end{table}

\bigskip
\bigskip

\vfill\eject

\end{document}